\documentclass[a4paper]{article}
\usepackage[utf8]{inputenc}
\usepackage[T1]{fontenc}
\usepackage[english]{babel}
\usepackage{bbm}
\usepackage{paralist}

\usepackage{amsmath,amsthm,amsfonts,amssymb}
\newtheorem{remark}{Remark}

\usepackage{stmaryrd}
\usepackage{subfigure}
\usepackage{graphicx}
\usepackage{multicol}
\usepackage{multirow}
\usepackage{paralist}
\usepackage{fontawesome}
\usepackage{tabularray}
\UseTblrLibrary{booktabs}
\usepackage{textcomp}
\usepackage[ruled]{algorithm2e}
\usepackage{comment}
\usepackage{copyrightbox}
\usepackage[nopatch=footnote]{microtype}
\usepackage{indentfirst}
\usepackage{lmodern}
\usepackage{stmaryrd}
\usepackage{fixmath}
\usepackage{csvsimple}
\usepackage{numprint}
\SetSymbolFont{stmry}{bold}{U}{stmry}{m}{n}
\usepackage{textgreek}
\usepackage{siunitx}
\DeclareSIUnit\mmHg{mmHg}
\usepackage{currfile}
\usepackage{import}
\usepackage{import}
\usepackage{xifthen}
\usepackage{pdfpages}
\usepackage{transparent}
\usepackage{booktabs}
\usepackage[colorinlistoftodos,prependcaption,textsize=tiny]{todonotes}
\usepackage{regexpatch}
\makeatletter
\xpatchcmd{\@todo}{\setkeys{todonotes}{#1}}{\setkeys{todonotes}{inline, #1}}{}{}
\makeatother
\usepackage{fancyhdr}
\usepackage[left=1.5cm,right=1.5cm,top=2cm,bottom=2cm]{geometry}
\usepackage{adjustbox}

\usepackage{footnote}
\makesavenoteenv{tabular}

\usepackage{pgfplots}
\usepackage{pgfplotstable}
\pgfplotsset{compat=newest}

\definecolor{colorA}{RGB}{49,140,231}
\definecolor{colorB}{RGB}{238,16,16}
\definecolor{colorC}{RGB}{34,120,15}
\definecolor{colorD}{RGB}{136,77,167}
\definecolor{colorE}{RGB}{224,17,95}
\definecolor{colorF}{RGB}{204,85,0}
\definecolor{colorG}{RGB}{214,39,40}
\definecolor{colorH}{RGB}{44,160,44}
\definecolor{colorI}{RGB}{148,103,189}
\definecolor{colorJ}{RGB}{140,86,75}

\definecolor{crimson2143940}{RGB}{214,39,40}
\definecolor{darkgray176}{RGB}{176,176,176}
\definecolor{darkorange25512714}{RGB}{255,127,14}
\definecolor{forestgreen4416044}{RGB}{44,160,44}
\definecolor{lightgray204}{RGB}{204,204,204}
\definecolor{steelblue31119180}{RGB}{31,119,180}
\definecolor{mediumpurple148103189}{RGB}{148,103,189}

\definecolor{customdarkblue}{HTML}{0C2472}
\definecolor{customcyan}{HTML}{00BFFF}
\definecolor{customorange}{HTML}{F1892C}
\definecolor{custompurple}{HTML}{800080}
\definecolor{customgreen}{HTML}{00FF00}
\definecolor{customyellow}{HTML}{FFD700}

\newtheorem{theorem}{Theorem}

\newcommand{\feelpp}{Feel\nolinebreak\hspace{-.05em}\raisebox{.4ex}{\tiny\bf +}\nolinebreak\hspace{-.10em}\raisebox{.4ex}{\tiny\bf +}}

\makeatletter
\newcommand\addplotgraphicsnatural[2][]{%
    \begingroup
    \pgfqkeys{/pgfplots/plot graphics}{#1}%
    \setbox0=\hbox{\includegraphics{#2}}%
    \pgfmathparse{\wd0/(\pgfkeysvalueof{/pgfplots/plot graphics/xmax} - \pgfkeysvalueof{/pgfplots/plot graphics/xmin})}%
    \let\xunit=\pgfmathresult
    \pgfmathparse{\ht0/(\pgfkeysvalueof{/pgfplots/plot graphics/ymax} - \pgfkeysvalueof{/pgfplots/plot graphics/ymin})}%
    \let\yunit=\pgfmathresult
    \xdef\marshal{%
        \noexpand\pgfplotsset{unit vector ratio={\xunit\space \yunit}}%
    }%
    \endgroup
    \marshal
    \addplot graphics[#1] {#2};
}
\makeatletter
\newcommand\resetstackedplots{
    \pgfplots@stacked@isfirstplottrue
}
\makeatother

\makeatletter
\title{Mathematical modeling and simulation of coupled aqueous humor flow and temperature distribution in a realistic 3D human eye geometry}
\author{%
    \texorpdfstring{
        \large Thomas Saigre\textsuperscript{1$\dagger$}, Vincent Chabannes\textsuperscript{1}, Christophe Prud'homme\textsuperscript{1}, Marcela Szopos\textsuperscript{2}, \\
        \small\textsuperscript{1}Institut de Recherche Mathématique Avancée, UMR 7501 Université de Strasbourg et CNRS\\
        \small\textsuperscript{2}Université Paris Cité, CNRS, MAP5, F-75006 Paris, France\\
        \small\textsuperscript{$\dagger$}Corresponding author: \href{mailto:thomas.saigre@cemosis.fr}{\texttt{thomas.saigre@cemosis.fr}}
    }{Thomas Saigre, Christophe Prud'homme, Marcela Szopos, Vincent Chabannes}
}
\def\@subject{}
\date{09-01-2026}
\usepackage[bookmarksnumbered,bookmarksopen,linktocpage,]{hyperref}
\hypersetup{
    pdftitle={\@title},
    pdfauthor={\@author},
    pdfsubject={\@subject},
    pdfkeywords={mathematical and computational ophthalmology, finite element method, thermo-fluid dynamics, feelpp},
    linkcolor=blue,
    filecolor=blue,
    urlcolor =blue,
    citecolor=green!80!black,
}
\usepackage[capitalise]{cleveref}

\makeatother

\AtBeginEnvironment{tikzpicture}{\tracinglostchars=0\relax}

\newcommand{\vct}[1]{{\vec{#1}}}
\newcommand{\vctAlg}[1]{\mathbold{#1}}
\newcommand{\mat}[1]{\underline{\underline{#1}}}
\newcommand{\matAlg}[1]{\underline{\underline{\boldsymbol{#1}}}}
\renewcommand{\P}{\mathbb{P}}
\newcommand{\R}{\mathbb{R}}
\newcommand{\V}{\mathbold{V}}
\renewcommand{\d}{~\mathrm{d}}
\newcommand{\dx}{\d \vct{x}}

\newcommand{\review}[1]{#1}

\graphicspath{
    {./}
}

\usepackage{csquotes}
\usepackage[backend=biber, sorting=none, style=alphabetic, eprint=false]{biblatex}
\AtEveryBibitem{\clearlist{language}}
\AtEveryBibitem{\clearlist{urldate}}

\begin{document}

\thispagestyle{empty}

\maketitle

\begin{abstract}
    We present a comprehensive computational model to simulate the coupled dynamics of aqueous humor flow and heat transfer in the human eye.
    To manage the complexity of the model, we make significant efforts in meshing and efficient solution of the discrete problem using high-performance resources.
    The model accurately describes the dynamics of the aqueous humor in the anterior and posterior chambers and accounts for convective effects due to temperature variations.
    Results for fluid velocity, pressure, and temperature distribution are in good agreement with existing numerical results in the literature.
    Furthermore, the effects of postural changes and wall shear stress behavior are analyzed, providing new insights into the mechanical forces acting on ocular tissues.
    Overall, the present contribution provides a detailed three-dimensional simulation that enhances the understanding of ocular physiology and may contribute to further progress in clinical research and treatment optimization in ophthalmology.
\end{abstract}


\section{Introduction}
\label{sec:intro}

Understanding the behavior of the human eye is challenging, as it involves the study of the interaction between various physical phenomena, such as heat transfer, fluid dynamics, and tissue deformation.
The eye is a complex organ in which these phenomena are intricately linked and influence each other in ways that are not yet fully understood. For example, the flow of aqueous humor within the eye can affect intraocular pressure, which in turn can influence the overall health of the tissues.
Therefore, it is crucial to develop accurate and efficient computational models to simulate the multi-physics physiology of the ocular system.
These models must integrate the various physical processes at play, and capture the interactions between thermal regulation, fluid movements, and mechanical responses of the eye tissues.
As a result, researchers can gain a more comprehensive understanding of the underlying mechanisms of ocular physiology, pathology and therapeutic options.
Examples in the latter case include hyperthermia for eye tumors \cite{li_modeling_2010}, cell injection treatment to cure bullous keratopathy \cite{kinoshita_injection_2018}, or ocular drug delivery methods \review{\cite{bhandari_ocular_2021,sadeghi_mathematical_2024}}.
In addition, invasive measurements on human subjects are complex to perform and may result in inaccurate results \cite{rosenbluth_temperature_1977}.
Therefore, numerical simulations are a valuable tool for investigating the complex interactions between the different physical phenomena in the eye, providing insights that are difficult or even impossible to obtain experimentally.

In the present work, we focus on simulating the flow of the aqueous humor in the anterior and posterior chambers of the human eye and its coupling with the heat transfer inside the eyeball.
Previous studies have investigated several aspects of these complex interactions, as reviewed in \cite{guidoboni_mathematical_2019}. For instance, \cite{heys_boussinesq_2002, wang_fluid_2016, murgoitio-esandi_mechanistic_2023} modeled flow coupled with heat transfer in the anterior chamber (AC) and posterior chamber (PC), while \cite{sacco_role_2023} explored the impact of the intraocular  pressure on aqueous humor (AH) flow and drainage. Other works, such as \cite{canning_fluid_2002, ooi_simulation_2008, bhandari_effect_2020,kilgour_operator_2021,dvoriashyna_mathematical_2022} examined the thermo-fluid dynamics of AH flow in the AC with specified boundary conditions.
However, these studies often focused on simplified geometries or did not fully couple the heat transfer within the entire eyeball.

\review{Mathematical models of ocular drug delivery, which are a particularly challenging field of interest, have been extensively developed, see for instance~\cite{sadeghi_mathematical_2024} and the references therein.
These models can either derive parameter values from \emph{in vivo} data, using a top-down approach or predict concentration profiles \emph{in vivo} based on mechanistic parameters, using a bottom-up modeling approach. The latter category includes computational fluid dynamics approaches that can accurately capture details of fluid flow, pressure, temperature gradients and drug concentrations in realistic anatomical models. Among the many interesting contributions in this area, we mention here:
\begin{inparaenum}[\it (i)]
    \item the pioneering works ~\cite{missel_physiologically_2019} for topical ocular delivery, and ~\cite{missel_simulating_2010} for intravitreal drug injection;
    \item the enhanced versions~\cite{lamminsalo_extended_2018,lamminsalo_extended_2020} including more complex anterior elimination and gravity effects;
    \item the original contributions to modeling intravitreal injections, incorporating interspecies translation~\cite{missel_simulating_2012} or post-injection eye deflation~\cite{ruffini_drug_2024}. In what concerns other administration routes, such as intracameral injections or subconjunctival administration, mostly simplified top-down modeling approaches were derived~\cite{sadeghi_mathematical_2024}, due to their more occasional use and relatively sparse data that are currently available.
\end{inparaenum}
}


In this challenging context, our work develops a three-dimensional mathematical and computational model that simulates heat transfer throughout the whole human eyeball, coupled with the dynamic flow of AH in both the AC and PC.
The present contribution represents an extension of the work presented in \cite{saigre_model_2024}, where only heat transfer inside the human eyeball was investigated, with a particular focus on the parametric effects.
The model aims to provide deeper insights into the thermal environment of the eye and its interaction with fluid dynamics, which is essential for understanding the ocular physiopathology and improving treatment strategies such as drug delivery methods. Special attention is paid to the impact of postural orientation on flow recirculations within the eye. Preliminary results have been presented in \cite{saigre_coupled_2024_abstract}, and we here further extend our analysis and perform a thorough validation process of our findings.
Furthermore, the wall shear stress (WSS) generated by the AH flow is an important biomechanical factor, influencing ocular tissue health and potentially impacting the drainage pathways, which are relevant to conditions such as glaucoma~\cite{yamamoto_effect_2010,yang_unraveling_2022}. We therefore include in our development an original contribution, allowing accurate and efficient computations of the WSS in the anterior and posterior chambers.

Solving the coupled three-dimensional thermo-fluid dynamics model of the eye numerically poses significant computational challenges due to the complexity and nonlinearity of the governing equations.
The interaction between heat transfer and fluid flow requires solving large, sparse linear systems that can be computationally expensive and time-consuming.
To address these challenges, it is essential to employ adapted preconditioners that can enhance the convergence rate of iterative solvers. By using tailored preconditioning techniques such as GAMG or the Schur complement~\cite{elman_finite_2014}, we achieve faster and more stable solutions, enabling the simulation of more detailed and realistic models of ocular physiology.

The remainder of the article is organized as follows: \cref{sec:model-phys} is devoted to the description of the biophysical model of the human eye, focusing on the detailed geometrical representation and the mechanisms governing AH flow.
\cref{sec:comp-framework} presents the discretization techniques and computational framework employed, including the numerical methods and preconditioning strategies used to solve the coupled equations efficiently.
\cref{sec:results} gathers and discusses the results obtained from our simulations, highlighting the model’s capabilities and potential applications.
Finally, \cref{sec:conclusion} summarizes our findings and outlines future research directions in the field of mathematical and computational ophthalmology.

\paragraph*{Notations:}
In this document, the notation $\vct{v}$ (with an upper arrow) denotes a vector as a physical quantity, while $\vctAlg{v}$ (in boldface) denotes a vector from the algebraic standpoint.
In the same spirit $\mat{A}$ (underlined) denotes a matrix or a tensor as a physical quantity, while $\matAlg{A}$ (in boldface) denotes a matrix from the algebraic standpoint.

\section{Biophysical model}
\label{sec:model-phys}

In this section, we present the biophysical model of the human eye that we aim to simulate.
We describe the geometrical model of the eye and the biomechanical behavior involved in heat transfer and fluid dynamics.

\subsection{Geometrical model of the human eye}

We denote by $\Omega$ the domain of the human eye presented in \cref{fig:geo-eye}.
This domain is divided into ten subdomains, representing the different parts of the eye, such as the cornea, the lens, the vitreous body, the retina, \emph{etc.}, with various physical properties.
We denote by $\Omega_\text{AH}$ the anterior and posterior chambers of the eye (brown part in \cref{fig:geo-eye}),
which are filled with the \emph{aqueous humor} (AH), as shown in \cref{fig:ah}.
The geometry, generated from a Computer-Aided Design (CAD) file of a human eye, is based on a realistic human eye model, with dimensions consistent with the average human eye anatomy~\cite{saigre_mathematical_2024}.
Note that
\begin{inparaenum}[\it (i)]
    \item the region corresponding to the suspensory ligaments, hatched in \cref{fig:geo-eye}, is not considered in the model and is included in the vitreous body;
    \item the iris-lens channel, see \cref{fig:ah}, is really thin~\cite{dvoriashyna_aqueous_2018}, very difficult to measure accurately and assumed to be in the order of 5 to \SI{10}{\micro\meter}\review{,
    \item the geometrical model contains a thorough description of the back segment of the eye, stemming for previous work \cite{sala_ocular_2023}.}
\end{inparaenum}

\review{
The biophysical model proposed in the present work and described in the next section involves the coupling between fluid dynamics in the anterior and posterior chambers, coupled with heat transfer in the entire ocular globe. Although other complex biofluid mechanisms pertaining to the posterior segment of the eye have not yet been incorporated and are being postponed for future work, we still include these regions in our realistic geometry for two main reasons. First, the viscoelastic properties of the vitreous gel were identified as playing an important role in ocular drug delivery models ~\cite{sadeghi_mathematical_2024}. Second, previous works ~\cite{scott_finite_1988} and sensitivity analysis studies we performed in~\cite{saigre_mathematical_2024} pointed out the significant impact of blood temperature and flow on the heat distribution. Therefore, all the domains depicted in \cref{fig:geo-eye} are included in the geometrical model. Moreover, the necessary steps from the CAD description to the mesh are described in \cite{chabannes_3d_2024}, are available in open access~\cite{saigre_mesh_2024} and could serve as a basis for future work.}

\begin{figure}[!ht]
    \centering
    \subfigure[Geometrical model of the human eye.]{
        \label{fig:geo-eye}
        \def\svgwidth{0.5\columnwidth}
        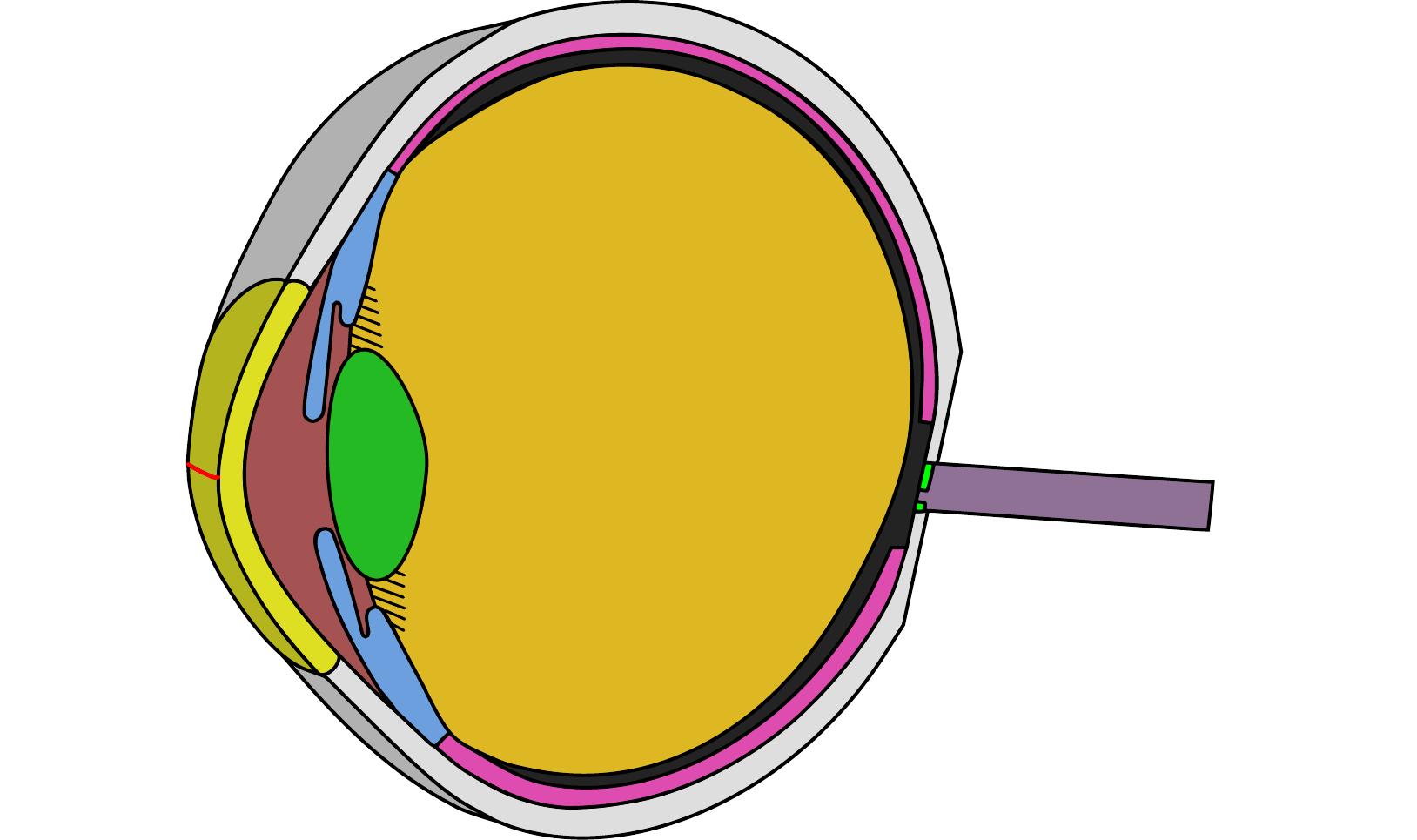
    }
    \subfigure[Vertical cut of the anterior and posterior chambers of the eye.]{
        \label{fig:ah}
        \def\svgwidth{0.3\columnwidth}
        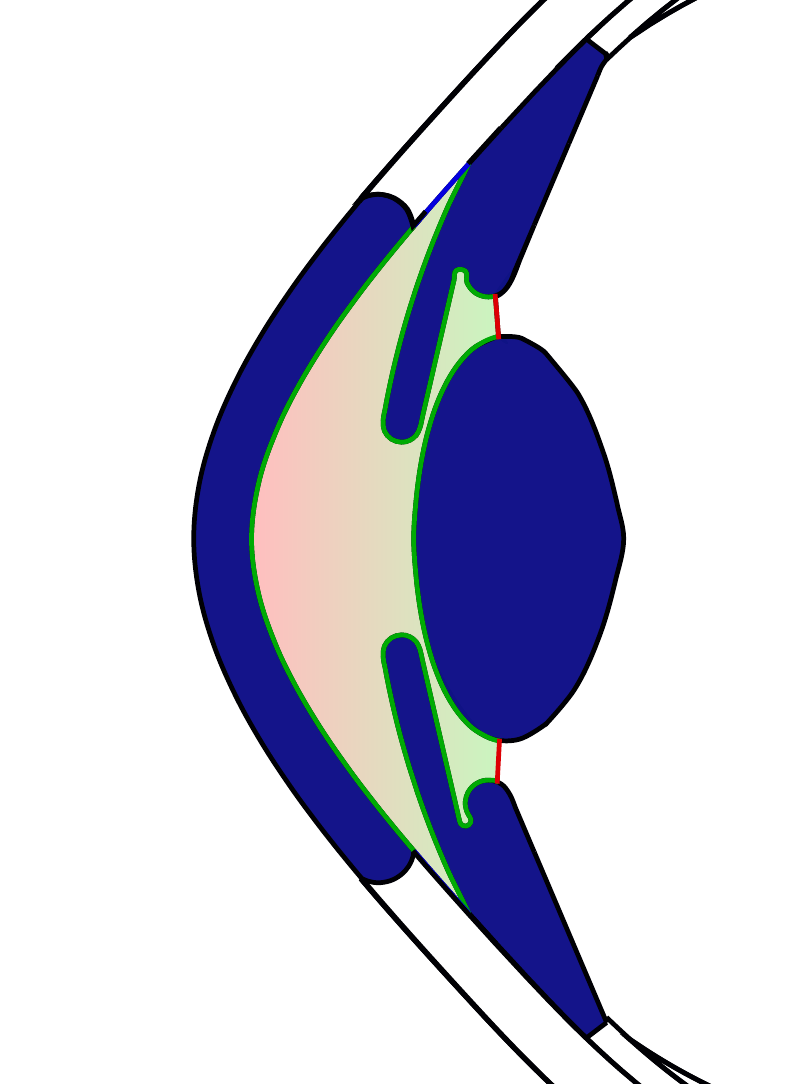
    }
    \caption{Description of the human eyeball (left panel), zoom on the anterior and posterior chambers of the eye (right panel).}
\end{figure}

\subsection{Bio-heat and fluid dynamics model}

The AH is a transparent fluid produced by the ciliary body that flows from the \emph{posterior chamber} (PC) to the \emph{anterior chamber} (AC), where it is drained through two pathways, the trabecular meshwork, and the uveoscleral pathway, see \cref{fig:ah-flow-anatomy} for a simplified view.
AH flow plays a fundamental role in maintaining the intraocular pressure (IOP)\index{IOP} level.
In addition to the hydraulic pressure difference created by production and drainage, the AH dynamic is influenced by posture and thermal factors.
Specifically, convective effects are produced by the temperature difference between the external environment at the corneal surface and the internal surface, which is at the body temperature.

\begin{figure}[!ht]
    \centering
    \def\svgwidth{0.4\columnwidth}
\begingroup%
  \makeatletter%
  \providecommand\color[2][]{%
    \errmessage{(Inkscape) Color is used for the text in Inkscape, but the package 'color.sty' is not loaded}%
    \renewcommand\color[2][]{}%
  }%
  \providecommand\transparent[1]{%
    \errmessage{(Inkscape) Transparency is used (non-zero) for the text in Inkscape, but the package 'transparent.sty' is not loaded}%
    \renewcommand\transparent[1]{}%
  }%
  \providecommand\rotatebox[2]{#2}%
  \newcommand*\fsize{\dimexpr\f@size pt\relax}%
  \newcommand*\lineheight[1]{\fontsize{\fsize}{#1\fsize}\selectfont}%
  \ifx\svgwidth\undefined%
    \setlength{\unitlength}{696.00002018bp}%
    \ifx\svgscale\undefined%
      \relax%
    \else%
      \setlength{\unitlength}{\unitlength * \real{\svgscale}}%
    \fi%
  \else%
    \setlength{\unitlength}{\svgwidth}%
  \fi%
  \global\let\svgwidth\undefined%
  \global\let\svgscale\undefined%
  \makeatother%
  \begin{picture}(1,0.76077586)%
    \lineheight{1}%
    \setlength\tabcolsep{0pt}%
    \put(0,0){\includegraphics[width=\unitlength,page=1]{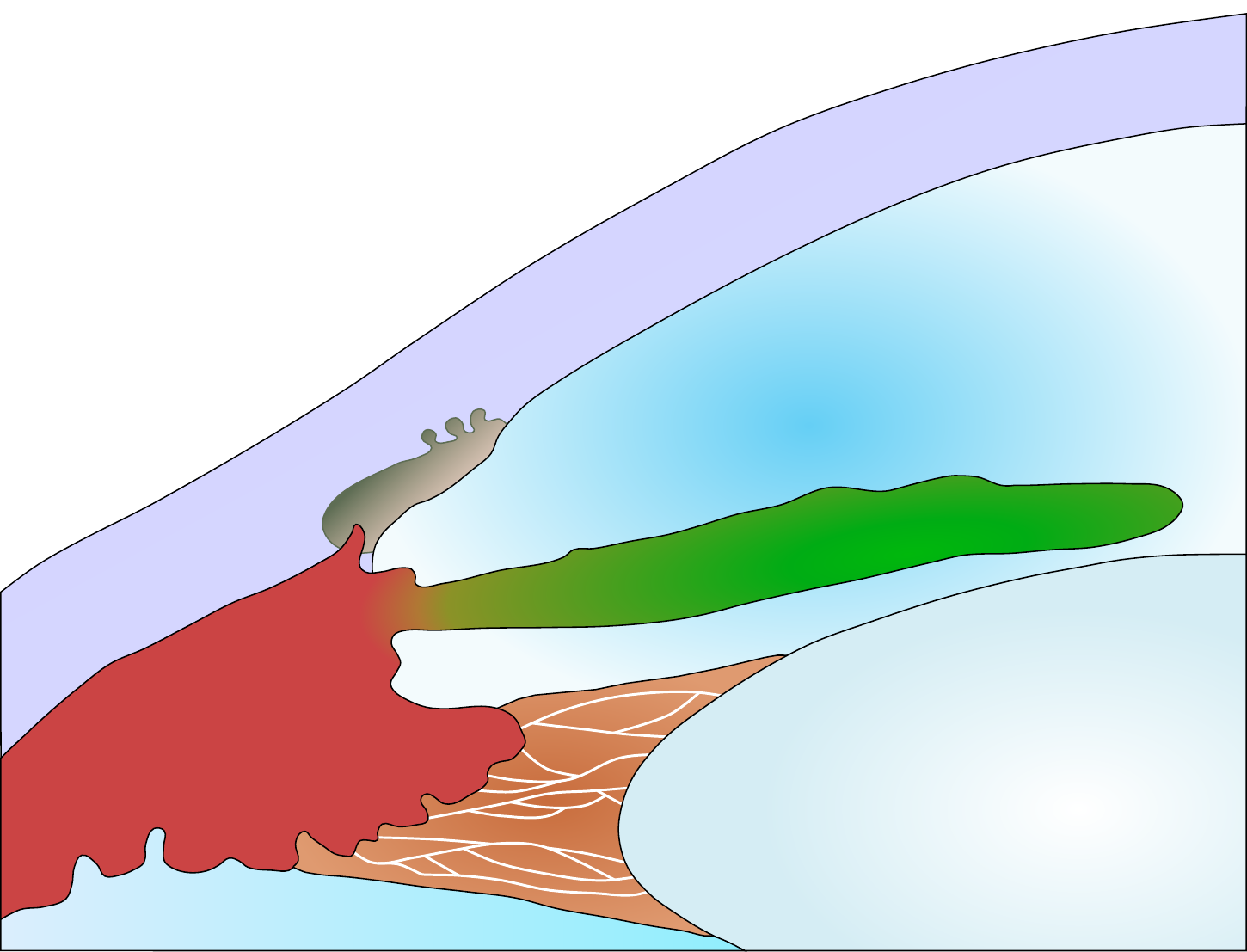}}%
    \put(0.85422689,0.12){\color[rgb]{0,0,0}\makebox(0,0)[t]{\lineheight{1.25}\smash{\begin{tabular}[t]{c}Lens\end{tabular}}}}%
    \put(0.28,0.48){\color[rgb]{0,0,0}\makebox(0,0)[t]{\lineheight{1.25}\smash{\begin{tabular}[t]{c}Trabecular\\meshwork\end{tabular}}}}%
    \put(0.58608856,0.30){\color[rgb]{0,0,0}\makebox(0,0)[t]{\lineheight{1.25}\smash{\begin{tabular}[t]{c}Iris\end{tabular}}}}%
    \put(0.6290214,0.59512269){\rotatebox{28}{\color[rgb]{0,0,0}\makebox(0,0)[t]{\lineheight{1.25}\smash{\begin{tabular}[t]{c}Cornea\end{tabular}}}}}%
    \put(0.79653537,0.46789173){\color[rgb]{0,0,0}\makebox(0,0)[t]{\lineheight{1.25}\smash{\begin{tabular}[t]{c}Anterior chamber\end{tabular}}}}%
    \put(0.18,0.14047472){\color[rgb]{0,0,0}\makebox(0,0)[t]{\lineheight{1.25}\smash{\begin{tabular}[t]{c}Ciliary body\end{tabular}}}}%
    \put(0,0){\includegraphics[width=\unitlength,page=2]{img-eye-ah_flow.pdf}}%
  \end{picture}%
\endgroup%

    \caption{Production and drainage of AH in the front part of the eye, adapted from \cite{ramakrishnan_aqueous_2013}.}
    \label{fig:ah-flow-anatomy}
\end{figure}

Following the approach of \citeauthor{kilgour_operator_2021}~\cite{kilgour_operator_2021} and \citeauthor{wang_fluid_2016}~\cite{wang_fluid_2016} (see also \cite{guidoboni_mathematical_2019} and references therein for a complete review), we make the following assumptions:
\begin{inparaenum}[\it (i)]
    \item The aqueous humor is considered an incompressible Newtonian fluid due to its low compressibility and viscosity.
    \item Density variations in the fluid are small and can be neglected except in the buoyancy term. This approximation, named \emph{Boussinesq approximation}, allows us to model the buoyancy effects due to temperature differences without accounting for full density variations.
\end{inparaenum}

Under these assumptions, the steady flow of AH is governed by the incompressible Navier-Stokes equations coupled with heat transfer:

\begin{subequations}
    \begin{alignat}{2}
        \rho (\vct{u}\cdot\nabla) \vct{u} - \nabla \cdot \mat{\sigma} &= -\rho\beta(T-T_\text{ref})\vct{g} &\quad & \text{in $\Omega_\text{AH}$}, \label{eq:pbStokes:stokes}\\
        \nabla\cdot \vct{u}                         &= 0            &      & \text{in $\Omega_\text{AH}$}, \label{eq:pbStokes:uncom}\\
        \rho C_p \vct{u}\cdot\nabla T - k\nabla^2 T &= 0            &      & \text{in $\Omega$}, \label{eq:pbStokes:heat}
    \end{alignat}
\label{eq:pbStokes}
\end{subequations}%
\unskip%
where \review{$\rho$ [\si{\kilogram\per\meter^3}] is the density of the fluid at reference temperature} $T_\text{ref}$ [\si{\kelvin}], $C_p$ [\si{\joule\per\kilogram\per\kelvin}] its specific heat, and $k$ [\si{\watt\per\meter\per\kelvin}] its thermal conductivity.
Note that $k$ is a discontinuous piecewise constant function because of the different materials in the eye.
We refer to \cite[Chap. 1]{saigre_mathematical_2024} for more details.
The quantity $T$ [\si{\kelvin}] is the temperature of the eye, while $p$ [\si{\pascal}] is the pressure of the aqueous humor fluid (also expressed in \si{\mmHg} in a biologic context),
and $\vct{u}$ [\si{\meter\per\second}] is its velocity.
The behavior of the fluid is characterized by the Cauchy stress tensor\index{Cauchy stress tensor} $\mat{\sigma}$ defined as
\begin{equation}
    \mat{\sigma}(\vct{u}, p) = -p\mat{I} + 2\mu\mat{D}(\vct{u}),
    \label{eq:pbStokes:stress}
\end{equation}
where \review{$\mu$ [\si{\newton.\second\per\meter^2}] is the dynamic viscosity of the fluid}, $\mat{I}$ is the identity tensor, and $\mat{D}(\vct{u}) = \frac12\left(\nabla \vct{u} + \nabla \vct{u}^T\right)$ is the \emph{strain rate tensor}\index{strain rate tensor}.
Following \cite{canning_fluid_2002,ng_fem_2006}, we neglect the metabolic heat generation in the eye due to blood perfusion as there is no literature data available on this topic.
However, such a term could be added in further analysis to enhance the model's accuracy.

The right-hand side term in \cref{eq:pbStokes:stokes} represents the gravitational force per unit volume, along with the Boussinesq approximation~\cite{drazin_hydrodynamic_2004},
utilized to account for the buoyancy effects due to temperature variations.
This approximation states that the fluid's density varies with temperature but remains virtually unaffected by pressure, as discussed above.
The coefficient $\beta$ [\si{\kelvin^{-1}}] is the fluid volume expansion coefficient, and $\vct{g}$ [\si{\meter.\second^{-2}}] the gravitational acceleration vector.
As the variation of temperature is small, we can consider that the density is constant in the fluid domain, and the Boussinesq approximation is valid.
Depending on the position of the patient (standing, laying supine or prone respectively), $\vct{g}$ can be either vertical ($\vct{g}=\left[0, -g, 0\right]^T$) or horizontal ($\vct{g}=\left[g, 0, 0\right]^T$, $\vct{g}=\left[-g, 0, 0\right]^T$ respectively), where $g$ is the gravitational acceleration.

We impose no-slip boundary conditions for the fluid velocity on the boundaries of $\Omega_\text{AH}$:
\begin{alignat}{2}
    \vct{u} &= \vct{0} &\quad & \text{on }\Gamma_C\cup \Gamma_I\cup \Gamma_L\cup\Gamma_\text{VH}\cup\Gamma_\text{Sc}.
    \label{eq:BcUp}
\end{alignat}
These boundary conditions, together with \cref{eq:pbStokes}, model the AH flow driven by thermal and gravitational effects.
The hydraulic pressure difference due to the production and drainage of AH is not explicitly accounted for, as previous studies have indicated that buoyancy is the dominant mechanism driving convective motion in the AH regardless of postural orientation~\cite{ooi_simulation_2008,kumar_numerical_2006}.

Note that alternative boundary conditions, such as non-homogeneous Dirichlet conditions~\cite{heys_boussinesq_2002,kilgour_operator_2021} or specified flow rates and pressures~\cite{wang_fluid_2016}, could also be considered but are beyond the scope of the present study.

Moreover, the heat transfer inside the eye is governed by the heat equation \eqref{eq:pbStokes:heat}, with the following boundary conditions,
taking into account the convective heat transfer with the surrounding tissues over a domain denoted as $\Gamma_\text{body}$; and the heat production due to the metabolism of the eye,
as well as the heat transfer with the ambient air over a domain called $\Gamma_\text{amb}$:
\begin{subequations}
\begin{alignat}{2}
    -k\textstyle\frac{\partial T}{\partial \vct{n}} &= h_\text{bl}(T-T_\text{bl}) &\quad & \text{on }\Gamma_\text{body}, \label{eq:BcT:BCTbody}\\
    -k\textstyle\frac{\partial T}{\partial \vct{n}} &= \underbrace{h_\text{amb}(T-T_\text{amb})}_\textit{(i)} + \underbrace{\sigma\varepsilon(T^4-T_\text{amb}^4)}_\textit{(ii)}+ \underbrace{E\vphantom{(h_\text{amb})}}_\textit{(iii)} &\quad & \text{on }\Gamma_\text{amb}, \label{eq:BcT:BCTamb}
\end{alignat}
\label{eq:BcT}
\end{subequations}
\unskip%
where $\Gamma_\text{body}$ and $\Gamma_\text{amb}$ are the boundaries of the eye in contact with the surrounding body and the ambient air, respectively.

Precisely, \cref{eq:BcT:BCTbody} models the convective heat transfer between the eye and the surrounding body,
where $h_\text{bl}$ [\si{\watt\per\meter\squared\per\kelvin}] is the heat transfer coefficient between the eye and the surrounding body, and $T_\text{bl}$ [\si{\kelvin}] is the blood temperature. \review{
Note that the impact of blood flow on the temperature could also being involved in other regions of the eye, modeled by a perfusion and a source term. Following a similar approach to ~\cite{ng_fem_2006}, we neglect the contribution of these terms in the present work. A more in-depth analysis of this assumption would be of interest, but is currently not available, mainly due to lack of literature data.}
On the other hand, \cref{eq:BcT:BCTamb} models three types of exchanges that are involved between the eyeball and the ambient air:
\begin{inparaenum}[\it (i)]
    \item Convective heat transfer, where $h_\text{amb}$ [\si{\watt\per\meter\squared\per\kelvin}] is the heat transfer coefficient between the eye and the ambient air, and $T_\text{amb}$ [\si{\kelvin}] is the ambient temperature;
    \item radiative heat transfer, where $\sigma = \qty{5.67e-8}{\watt\per\meter\squared\per\kelvin^4}$ represents the Stefan-Boltzmann constant, and $\varepsilon$ [--] is the emissivity of the cornea; and
    \item evaporative heat loss, due to tear evaporation at the surface of the eye, where $E$ [\si{\watt\per\meter\squared}] is the evaporation rate of the tear film.
\end{inparaenum}

\begin{remark}
    \label{rem:stokes}
    The momentum \cref{eq:pbStokes:stokes} could be further simplified by discarding the contribution of the non-linear term $\rho(\vct{u}\cdot\nabla)\vct{u}$, and considering a Stokes flow, since the Reynolds number is reported ti be relatively small in this case, see for instance \cite{wang_fluid_2016}.
    However, due to the coupling with the heat equation in the system, an iterative approach is still required to solve the problem.
    Moreover, several other works \cite{kilgour_operator_2021,wang_fluid_2016} consider the full Navier-Stokes equations, and therefore we utilized the same model for consistency and comparison.
    We performed a comparative study of different modeling approaches, see results in \cref{rmk:stokes:ns}.
\end{remark}

\subsection{\review{Wall shear stress analysis}}

We further focus on the wall shear stress (WSS), a critical parameter representing the tangential force per unit area exerted by the fluid on the wall due to viscous effects.
The insights gained from WSS analysis can inform clinical practices, contributing to personalized medicine, surgical optimization, and improved device design~\cite{yang_unraveling_2022,fernandez-vigo_computational_2018}.
Understanding the WSS distribution has significant implications in ocular physiology, including:

\begin{itemize}
    \item \textbf{Drug delivery:} High WSS regions may enhance the mixing and transport of drug particles within the AH, potentially increasing drug absorption rates through ocular tissues, see~\cite{xu_potential_2013,koutsiaris_wall_2016,saxer_use_2013}.
    By identifying these regions, drug delivery systems can be designed to target specific areas within the eye, improving therapeutic outcomes.

    \item \textbf{\review{Therapeutical} procedures:} WSS analysis provides valuable insights for optimizing \review{therapeutical}  interventions, \review{such as:}
\begin{inparaenum}[\it (i)]
    \item Design optimization: Knowledge of WSS distributions helps in planning \review{surgical} procedures that minimize mechanical stresses on ocular tissues, reducing the risk of tissue damage, \cite{kudsieh_numerical_2020};
    \item Implantable devices: Designing intraocular lenses and drainage devices that account for WSS can help prevent endothelial cell loss and maintain corneal health~\cite{repetto_phakic_2015,basson_computational_2024};
    \item Postoperative outcomes: Monitoring changes in WSS after \review{ocular therapies, such as endothelial cell injection \cite{kinoshita_injection_2018},} can help predict healing responses and the risk of complications.
\end{inparaenum}
\end{itemize}

\paragraph*{Definition of Wall Shear Stress}
The \emph{wall shear stress}\index{wall shear stress} $\vct{\tau}_w$ at a point on the wall is defined as the magnitude of the tangential component of the stress tensor acting on the wall, and is computed in the present work as:

\begin{equation}
    \vct{\tau}_w(\vct{u}, p) = \mu \left. \frac{\partial \vct{u}_\tau}{\partial \vct{n}} \right|_{\text{wall}},
    \label{eq:wss_simplified}
\end{equation}
under the no-slip boundary condition at the wall.
The vector $\vct{n}$ designates the unit outward normal vector to the wall,
and $\vct{u}_\tau$ is the tangential component of the velocity vector at the wall.

\section{Mathematical and computational framework}
\label{sec:comp-framework}

In this section, we present the mathematical and computational framework developed to solve the coupled fluid dynamics and heat transfer model described in \cref{eq:pbStokes,eq:BcUp,eq:BcT}.
We detail the variational formulation in \cref{sec:comp-framework:cont}, the geometrical and finite element discretization in Sections \ref{sec:comp-geo} and \ref{sec:coupled:fe} respectively, and the solution strategy employed in \cref{sec:comp-strategy}.

\subsection{Continuous setting}
\label{sec:comp-framework:cont}

We begin by deriving the variational formulation of the model. Let us introduce the following functional spaces:

\begin{enumerate}[\itshape (i)]
    \item The velocity space $\V := [H^1_0(\Omega_\text{AH})]^3$, consisting of vector fields $\vct{u}$ with square-integrable derivatives with a trace that vanishes on the boundary;
    \item the pressure space $Q := L^2_0(\Omega_\text{AH}) = \left\{ p \in L^2(\Omega_\text{AH}) \middle| \int_{\Omega_\text{AH}} p \, \mathrm{d}\vct{x} = 0 \right\}$, consisting of square-integrable scalar fields with zero mean; and
    \item the temperature space $W := H^1(\Omega)$.
\end{enumerate}

Let $\vct{v} \in \V$, $q \in Q$, and $\varphi \in W$ be test functions.
Multiplying Equations \eqref{eq:pbStokes:stokes}, \eqref{eq:pbStokes:uncom}, and \eqref{eq:pbStokes:heat} by $\vct{v}$, $q$, and $\varphi$, respectively, and integrating by part over the appropriate domains yield:

\begin{subequations}
\begin{gather}
    \nonumber
    \rho\int_{\Omega_\text{AH}} (\vct{u} \cdot \nabla )\vct{u} \cdot \vct{v} \dx + \mu \int_{\Omega_\text{AH}} \mat{D}(\vct{u}) : \nabla \vct{v} \dx - \int_{\Omega_\text{AH}} p \cdot \nabla \vct{v} \dx\\
        \qquad\qquad\qquad\qquad\qquad\qquad + \int_{\Omega_\text{AH}} \rho_0 \beta T\vct{g} \cdot \vct{v} \dx
        = \int_{\Omega_\text{AH}} \rho_0 \beta T_\text{ref} \vct{g} \cdot \vct{v} \dx, \\
    \int_{\Omega_\text{AH}} \nabla \cdot \vct{u} q \dx = 0, \\
    \rho C_p \int_{\Omega} \vct{u} \cdot \nabla T \varphi \dx + k \int_{\Omega} \nabla T \cdot \nabla \varphi \dx
        + \int_{\Gamma_\text{amb}} \left(h_\text{amb}T \d\sigma + \sigma\varepsilon T^4\right)\varphi \d\sigma
        + \int_{\Gamma_\text{body}} h_\text{bl} T\varphi \d\sigma \nonumber\\
        \qquad\qquad\qquad\qquad\qquad\qquad= \int_{\Gamma_\text{amb}} \left(h_\text{amb}T_\text{amb} + \sigma\varepsilon T^4_\text{amb}\right)\varphi \d\sigma
         + \int_{\Gamma_\text{body}} h_\text{bl} T_\text{bl}\varphi \d\sigma.
\end{gather}
\end{subequations}
We define the following bilinear and trilinear forms:

\begin{subequations}
\begin{gather}
    a_1(\vct{z}, \vct{u}, \vct{v}) = \rho\int_{\Omega_\text{AH}} (\vct{z} \cdot \nabla) \vct{u} \cdot \vct{v} \dx,\\
    a_2(\vct{u}, \vct{v}) = \mu \int_{\Omega_\text{AH}} \mat{D}(u) : \nabla \vct{v} \dx,\\
    b(p, \vct{v}) = -\int_{\Omega_\text{AH}} p \cdot \nabla \vct{v} \dx,\\
    d(T, \vct{v}) = -\int_{\Omega_\text{AH}} \rho_0 \beta T \vct{g} \cdot \vct{v} \dx,\\
    e(\vct{u}, T, \varphi) = \rho C_p\int_{\Omega} \vct{u}\cdot\nabla T \varphi \dx,\\
    f(T, \varphi) = k\int_{\Omega} \nabla T\cdot\nabla\varphi \dx + \int_{\Gamma_\text{amb}} \left(h_\text{amb}T + \sigma\varepsilon T^4\right)\varphi \d\sigma
        + \int_{\Gamma_\text{body}} h_\text{bl} T\varphi \d\sigma, \\
    \ell_1(\vct{v}) = \int_{\Omega_\text{AH}} \rho_0 \beta T_\text{ref} \vct{g} \cdot \vct{v} \dx, \\
    \ell_2(\varphi) = \int_{\Gamma_\text{amb}} \left(h_\text{amb}T_\text{amb} + \sigma\varepsilon T^4_\text{amb}\right)\varphi \d\sigma
         + \int_{\Gamma_\text{body}} h_\text{bl} T_\text{bl}\varphi \d\sigma.
\end{gather}
\end{subequations}

The variational formulation of the problem is then: Find $(\vct{u}, p, T) \in \V \times Q \times W$ such that for all $(\vct{v}, q, \varphi) \in \V \times Q \times W$:

\begin{alignat}{2}
    a_1(\vct{u},\vct{u}, \vct{v}) + a_2(\vct{u},\vct{v}) &+ b(p, \vct{v}) &+ d(T, \vct{v}) &= \ell_1(\vct{v}), \nonumber \\
    b(q, \vct{u}) & & &= 0, \label{eq:pb-var} \\
    e(\vct{u}
    , T, \varphi) & &+ f(T, \varphi) &= \ell_2(\varphi). \nonumber
\end{alignat}

\begin{theorem}[Existence and Uniqueness]
    The variational problem \eqref{eq:pb-var} has a unique solution.
\end{theorem}

\begin{proof}
    A detailed proof is beyond the scope of this work.
    However, the existence and uniqueness can be established using fixed-point arguments and standard results for the Navier-Stokes equations coupled with heat transfer, as discussed in \cite{tsuzuki_existence_2015}.
\end{proof}

\subsection{Geometrical discretization}
\label{sec:comp-geo}

The geometry presented in \cref{fig:geo-eye} is discretized, with significant effort devoted to the AC and PC domains, where the coupled fluid-thermal problem is solved.
This meticulous approach ensures that the complexities of these regions are accurately captured, leading to more precise simulation results. The discretization results in a mesh composed of \pgfmathprintnumber{4967553} elements, highlighting the level of detail and refinement applied to the model, see ~\cref{fig:coupled:wss_mesh_refinement} and ~\cref{fig:temp}.
Such detailed refinement is essential for resolving the interactions between fluid flow and thermal dynamics within these areas, thereby improving the overall accuracy and reliability of the simulation outcomes.
A few remarks are in order in practice:
\begin{inparaenum}[\it (i)]
\item the mesh generation is achieved \textit{via} \cite{cascade_salome_2022},
\item the mesh adaptations and refinements are performed using \cite{mmg_tools_mmg_2022} and the level set features of \feelpp~\cite{prudhomme_feelppfeelpp_2024},
\item the mesh is not only refined to obtain a hierarchy of increasingly refined meshes but also adapted along the boundaries to obtain a graded mesh with respect to the distance to the boundaries, thanks to the level set features above-mentioned.
The latter feature is particularly important to obtain accurate surface quantities' approximation, such as the wall shear stress.
\end{inparaenum}

A complete description of the mesh generation and adaptation process is available in \cite{chabannes_3d_2024}.
Moreover, the family of meshes is available publicly~\cite{saigre_mesh_2024}.

\subsection{Finite element setting}
\label{sec:coupled:fe}

We discretize the variational problem using the finite element method (FEM).
The computational domain is discretized as described in \cite{chabannes_3d_2024}, with mesh refinement in the anterior and posterior chambers to accurately capture the flow and thermal dynamics.

We define the finite element spaces for velocity, pressure, and temperature as follows:

\begin{itemize}
    \item $\V_h := \left\{\vct{v}\in[\P_2(\Omega_\text{AH})]^3 \,\middle|\, \vct{v} = \vct{0}\text{ on }\partial\Omega\right\}$, the space of vector-valued piecewise quadratic polynomials that vanish on the boundary for velocity;
    \item $Q_h := \P_1(\Omega_\text{AH})$, the space of piecewise linear polynomials for pressure; and
    \item $W_h := \P_1(\Omega)$, the space of piecewise linear polynomials for temperature.
\end{itemize}

This choice corresponds to the $\P_2\P_1\text{--}\P_1$ finite element spaces, where $\P_2\P_1$ is the Taylor-Hood element for the velocity-pressure pair, which satisfies the LBB (Ladyzhenskaya-Babu\v{s}ka-Brezzi) stability condition~\cite{elman_finite_2014}.
Other discretization strategies could be employed, such as $\P_1\P_1\text{--}\P_1$ or $\P_2\P_1\text{--}\P_2$.
The former requires a stabilization term to ensure stability, while the latter is more computationally expensive and may not provide significant improvements in accuracy.

The discrete problem reads: Find $(\vct{u}_h, p_h, T_h) \in \V_h \times Q_h \times W_h$ such that for all $(\vct{v}_h, q_h, \varphi_h) \in \V_h \times Q_h \times W_h$:

\begin{equation}
\left\{
\begin{alignedat}{2}
    a_1(\vct{u}_h,\vct{u}_h, \vct{v}_h) + a_2(\vct{u}_h, \vct{v}_h) &+ b(p_h, \vct{v}_h) &+ d(T_h, \vct{v}_h) &= \ell_1(\vct{v}_h), \\
    b(q_h, \vct{u}_h) & & &= 0,\\
    e(\vct{u}_h, T_h, \varphi_h) & &+ f(T_h, \varphi_h) &= \ell_2(\varphi_h).
\end{alignedat}
\right.
\label{eq:pb-var-disc}
\end{equation}

The system \eqref{eq:pb-var-disc} is non-linear because of the terms $a_1$ and $e$.
We employ Newton's method to solve it iteratively.
Precisely, it consists of starting from an initial guess $(\vct{u}^0, p^0, T^0)$, and iteratively compute $(\vct{u}^{k+1}, p^{k+1}, T^{k+1})$ as the solution of the non-linear system at each iteration.
We set the correction terms $\delta \vct{u}^k := \vct{u}^{k+1} - \vct{u}^k$, $\delta p^k := p^{k+1} - p^k$, and $\delta T^k := T^{k+1} - T^k$.
Given $(\vct{u}^k, p^k, T^k)$, we define the \emph{nonlinear residual} associated with the variational formulation \eqref{eq:pb-var-disc} as:

\begin{equation}
\left\{
\begin{alignedat}{2}
    r_{\vct{u}}^k(\vct{v}) &:= \ell_1(\vct{v}) - a_1(\vct{u}^k, \vct{u}^k, \vct{v}) - a_2(\vct{u}^k, \vct{v}) - b(p^k, \vct{v}) - d(T^k, \vct{v}), \\
    r_p^k(q)               &:= -b(q, \vct{u}^k), \\
    r_T^k(\varphi)         &:= \ell_2(\varphi) - e(\vct{u}^k, T^k, \varphi) - f(T^k, \varphi).
\end{alignedat}
\right.
\end{equation}

Following \cite[Ch.~8]{elman_finite_2014}, by dropping the quadratics terms, the correction terms verify the following weak linear problem:
$\forall (\vct{v}, q, \varphi) \in \V \times Q \times T$:

\begin{equation}
\left\{
\begin{alignedat}{2}
    a_1(\delta \vct{u}^k, \vct{u}^k, \vct{v}) + a_1(\vct{u}^k, \delta \vct{u}^k, \vct{v}) + a_2(\delta \vct{u}^k, \vct{v}) + b(\delta p^k, \vct{v}) + d(\delta T^k, \vct{v}) &= r_{\vct{u}}^k(\vct{v}), \\
    b(\delta p^k, \vct{u}^k) &= r_p^k(q), \\
    e(\delta \vct{u}^k, T^k, \varphi) + e(\vct{u}^k, \delta T^k, \varphi) + f(\delta T^k, \varphi) &= r_T^k(\varphi).
\end{alignedat}
\right.
\end{equation}

In the discretized spaces $\V_h\times Q_h\times T_h$, the discrete counterpart of this variational problem is solved.
To define the corresponding linear algebra problem, we set the basis of the discrete spaces: $\{\vct{\lambda}_i\}_{i=1}^{N_u}$ is a basis of $\V_h$, $\{\mu_j\}_{j=1}^{N_p}$ is a basis of $Q_h$, and $\{\xi_l\}_{l=1}^{N_T}$ is a basis of $T_h$.
Setting $\vctAlg{u}$, $\vctAlg{\Delta u}$, $\vctAlg{p}$, $\vctAlg{\Delta p}$, $\vctAlg{T}$, and $\vctAlg{\Delta T}$ the vectors of the coefficients of the basis functions in the corresponding basis of $\vec{u}$, $\delta \vec{u}$, $p$, $\delta p$, $T$, and $\delta T$ respectively, the algebraic problem reads:

\begin{equation}
    \begin{bmatrix}
        \matAlg{V}^k + \matAlg{W}^k + \matAlg{N} & \matAlg{B}^T & \matAlg{D} \\
        \matAlg{B}                               & \matAlg{0}   & \matAlg{0} \\
        \matAlg{E}_1^k                           & \matAlg{0}   & \matAlg{E}_2^k + \matAlg{F}
    \end{bmatrix}
    \begin{bmatrix}
        \vctAlg{\Delta u} \\
        \vctAlg{\Delta p} \\
        \vctAlg{\Delta T}
    \end{bmatrix}
    =
    \begin{bmatrix}
        \vctAlg{r}_{\vct{u}}^k \\
        \vctAlg{r}_p^k \\
        \vctAlg{r}_T^k
    \end{bmatrix},
    \label{eq:system-lin-alg}
\end{equation}
where all elements of this system are defined on the basis of the discrete spaces:
\begin{equation}
    \begin{gathered}
        \matAlg{V}^k           = \begin{bmatrix}a_1(\vct{u}^k, \vct{\lambda}_i, \vct{\lambda}_j)\end{bmatrix} \in \R^{N_u\times N_u}, \hspace{0.9em}
        \matAlg{W}^k           = \begin{bmatrix}a_1(\vct{\lambda}_i, \vct{u}^k, \vct{\lambda}_j)\end{bmatrix} \in \R^{N_u\times N_u}, \hspace{0.9em}
        \matAlg{N}             = \begin{bmatrix}a_2(\vct{\lambda}_i, \vct{\lambda}_j)\end{bmatrix}            \in \R^{N_u\times N_u}, \\
        \matAlg{B}             = \begin{bmatrix}b(\mu_l, \vct{\lambda}_j)\end{bmatrix}                        \in \R^{N_p\times N_u}, \quad
        \matAlg{D}             = \begin{bmatrix}d(\vct{\lambda}_i, \xi_l)\end{bmatrix}                        \in \R^{N_u\times N_T}, \\
        \matAlg{E}_1^k         = \begin{bmatrix}e(\vct{\lambda}_i, T^k, \xi_l)\end{bmatrix}                   \in \R^{N_u\times N_T}, \quad
        \matAlg{E}_2^k         = \begin{bmatrix}e(\vct{u}^k, \xi_l, \xi_m)\end{bmatrix}                       \in \R^{N_T\times N_T}, \quad
        \matAlg{F}             = \begin{bmatrix}f(\xi_l, \xi_m)\end{bmatrix}                                  \in \R^{N_T\times N_T}, \\
        \vctAlg{r}_{\vct{u}}^k = \begin{bmatrix}r_{\vct{u}}^k(\vct{\lambda}_i)\end{bmatrix}                   \in \R^{N_u}, \quad
        \vctAlg{r}_p^k         = \begin{bmatrix}r_p^k(\mu_j)\end{bmatrix}                                     \in \R^{N_p}, \quad
        \vctAlg{r}_T^k         = \begin{bmatrix}r_T^k(\xi_l)\end{bmatrix}                                     \in \R^{N_T}.
    \end{gathered}
    \label{eq:algebraic-form}
\end{equation}

From the initial guess, we iteratively solve the linear algebra problem \eqref{eq:system-lin-alg} at each Newton iteration, until the relative increment $\operatorname{Crit}^k := \operatorname{max}\{\|\delta \vct{u}^k\|/\|\delta \vct{u}^0\|,\|\delta p^k\|/\|\delta p^0\|, \|\delta T^k\|/\|\delta T^0\|\} \leqslant \varepsilon_{\operatorname{tol}}$ is reached for a given tolerance $\varepsilon_{\operatorname{tol}}>0$.
\cref{algo:fe-newton} summarizes the Newton iteration loop.

\begin{algorithm}
    \KwIn{$\{\vctAlg{u}^0, \vctAlg{p}^0, \vctAlg{T}^0,\varepsilon_{\operatorname{tol}}\}$.}
    $(\vctAlg{u}^0, \vctAlg{p}^0, \vctAlg{T}^0) \leftarrow$ initial guess\;
    Assemble $\matAlg{N}$, $\matAlg{B}$, $\matAlg{D}$, $\matAlg{F}$\;
    \While{$\operatorname{Crit}^k > \varepsilon_{\operatorname{tol}}$}{
        Assemble $\matAlg{V}^k$, $\matAlg{W}^k$, $\matAlg{E}_1^k$, $\matAlg{E}_2^k$, $\vctAlg{r}_{\vct{u}}^k$, $\vctAlg{r}_p^k$, $\vctAlg{r}_T^k$\;
        $(\vctAlg{\Delta u}, \vctAlg{\Delta p}, \vctAlg{\Delta T}) \leftarrow$ solution to System \eqref{eq:system-lin-alg}\;
        $\vctAlg{u}^{k+1} \leftarrow \vctAlg{u}^k + \alpha\vctAlg{\Delta u}$, $\vctAlg{p}^{k+1} \leftarrow \vctAlg{p}^k + \alpha\vctAlg{\Delta p}$, $\vctAlg{T}^{k+1} \leftarrow \vctAlg{T}^k + \alpha\vctAlg{\Delta T}$\;
    }
    \KwOut{$(\vctAlg{u}^{k+1}, \vctAlg{p}^{k+1}, \vctAlg{T}^{k+1})$.}
    \caption{Newton iteration loop.}
    \label{algo:fe-newton}
\end{algorithm}

The parameter $\alpha\in[0,1]$ is a relaxation factor that can be adjusted, using the line search method of PETSc~\cite{balay_petsc_2024}, to improve the convergence of the Newton iteration.

\subsection{Solution strategy}
\label{sec:comp-strategy}

We implement the computational framework using the \textsf{heatfluid} toolbox of \feelpp{}\footnote{See documentation: \url{https://docs.feelpp.org/toolboxes/latest/heatfluid/}}~\cite{prudhomme_feelppfeelpp_2024} to solve~\cref{algo:fe-newton}.
Efficiently solving the resulting linear systems at each Newton iteration is crucial.

Direct solvers become impractical for large-scale problems due to computational and memory constraints.
Therefore, we employ iterative solvers with appropriate preconditioners to enhance convergence.
The preconditioner is a key component in the iterative solver, as it actually enables the solver to converge.
We utilize a \emph{field-split} preconditioning strategy, where the global system is partitioned into smaller blocks corresponding to different physical fields, namely fluid and thermal fields.
This approach allows us to apply specialized solvers and preconditioners to each block.


For clarity, we rewrite the system \eqref{eq:system-lin-alg} with block notation, omitting the superscript $k$:

\begin{equation}
    \left[
    \begin{array}{cc|c}
        \matAlg{\widetilde{A}} & \matAlg{B}^T & \matAlg{D} \\
        \matAlg{B}             & \matAlg{0}   & \matAlg{0} \\
        \hline
        \matAlg{E}             & \matAlg{0}   & \matAlg{\widetilde{F}}
    \end{array}
    \right]
    \left[
    \begin{array}{cc}
        \vctAlg{\Delta u} \\
        \vctAlg{\Delta p} \\
        \hline
        \vctAlg{\Delta T}
    \end{array}
    \right]
    =
    \left[
    \begin{array}{cc}
        \vctAlg{r}_{\vct{u}} \\
        \vctAlg{r}_p \\
        \hline
        \vctAlg{r}_T
    \end{array}
    \right]
    \quad
    \Longleftrightarrow:
    \quad
    \underbrace{
    \left[
    \begin{array}{c|c}
        \matAlg{K}_{0,0} & \matAlg{K}_{0, 1} \\
        \hline
        \matAlg{K}_{1,0} & \matAlg{K}_{1, 1}
    \end{array}
    \right]
    }_{=:\matAlg{K}}
    \begin{bmatrix}
        \vctAlg{\Delta}_\text{fluid} \\
        \vctAlg{\Delta}_\text{heat}
    \end{bmatrix}
    =
    \begin{bmatrix}
        \vctAlg{r}_\text{fluid} \\
        \vctAlg{r}_\text{heat}
    \end{bmatrix}.
\end{equation}
The main idea of \emph{additive fieldsplit} preconditioner is to approximate the inverse of the matrix $\mat{K}$ by the matrix
\begin{equation}
    \begin{bmatrix}
        \matAlg{K}_{0,0}^{-1} & \matAlg{0} \\
        \matAlg{0}            & \matAlg{K}_{1,1}^{-1}
    \end{bmatrix},
\end{equation}
where the inverses of the diagonal blocks are applied separately, with appropriate solvers and associated preconditioners.

The heat block $\matAlg{K}_{1,1}$ inverse is approximated using a few iterations of \emph{GAMG} --- Geometric Algebraic Multigrid --- from PETSc~\cite{balay_petsc_2024}.
This preconditioner efficiently handles large sparse matrices by recursively coarsening and solving the problem on multiple levels, significantly accelerating the convergence.

On the other hand, concerning the fluid block $\matAlg{K}_{0,0}$, the inverse is approximated using the Schur complement, as proposed in \cite[Ch.~9]{elman_finite_2014}, where another field split preconditioner is implemented: the degrees of freedom are now divided into velocity and pressure blocks.
Then, the LDU decomposition of the matrix $\matAlg{K}_{0, 0}$ is computed as follows:

\begin{equation}
    \matAlg{K}_{0, 0} = \left[
        \begin{array}{c|c}
            \matAlg{\widetilde{A}} & \matAlg{B}^T \\
            \hline
            \matAlg{B}             & \matAlg{0}
        \end{array}
    \right]
    = \underbrace{\begin{bmatrix}
        \matAlg{I} & \matAlg{0} \\
        \matAlg{B}\,\matAlg{\widetilde{A}}^{-1} & \matAlg{I}
    \end{bmatrix}}_{\matAlg{L}}
    \underbrace{\begin{bmatrix}
        \matAlg{\widetilde{A}} & \matAlg{0} \\
        \matAlg{0}         & \matAlg{S}
    \end{bmatrix}}_{\matAlg{D}}
    \underbrace{\begin{bmatrix}
        \matAlg{I} & \matAlg{\widetilde{A}}^{-1}\matAlg{B}^T \\
        \matAlg{0} & \matAlg{I}
    \end{bmatrix}}_{\matAlg{U}},
    \label{eq:ldu}
\end{equation}
where $\matAlg{S} = -\matAlg{B}\,\matAlg{\widetilde{A}}^{-1}\matAlg{B}^T$ is the \emph{Schur complement operator}\index{Schur complement}.
We employ here an \emph{upper} strategy, namely we approximate the inverse of $\matAlg{K}_{0, 0}$ by the matrix $(\matAlg{D}\,\matAlg{U})^{-1}$:
\begin{equation}
    \matAlg{K}_{0,0}^{-1} \approx
    \begin{bmatrix}
        \matAlg{I} & -\matAlg{\widetilde{A}}^{-1}\matAlg{B}^T \\
        \matAlg{0} & \matAlg{I}
    \end{bmatrix}
    \begin{bmatrix}
        \matAlg{\widetilde{A}}^{-1} & \matAlg{0} \\
        \matAlg{0}                  & \matAlg{S}^{-1}
    \end{bmatrix}.
    \label{eq:approx-schur}
\end{equation}
Note that these two matrices are not computed explicitly, as we are computing the action of the inverse on a vector.
Finally, another fieldsplit is applied to the velocity block $\matAlg{\widetilde{A}}$, where the velocity components are solved separately using a block Jacobi preconditionner, decoupling the velocity components.

We summarize in \cref{tab:pc_summary} the preconditioners used for the different blocks of the system \eqref{eq:system-lin-alg}, as well as the Krylov Subspace Method (KSP) employed to solve the linear system.
These methods can be either the \emph{Generalized Minimal Residual} (GMRES) or the \emph{Flexible GMRES} (FGMRES) methods, depending on the block being solved.
In some cases, the \texttt{preonly} solver is used, which indicates that no Krylov subspace solver is executed, rather only the associated preconditioner is applied once.
More details on the setting of these preconditioners can be found in the configuration file presented in \cite{saigre_mesh_2024}.


\begin{table}
    \centering
    \begin{tabular}{cccc}
        \toprule
        \textbf{Block}       & \textbf{Description}    & \textbf{Preconditioner}      & \textbf{Krylov solver} \\
        \midrule
        $\matAlg{K}$         & Overall coupled problem & \texttt{fieldsplit additive} & \texttt{gmres} \\
        $\matAlg{K}_{1,1}$   & Heat split subproblem   & \texttt{gamg}                & \texttt{gmres} \\
        $\matAlg{K}_{0,0}$   & Fluid split subproblem  & \texttt{fieldsplit schur}    & \texttt{fgmres} \\
        $\matAlg{\widetilde{A}}$ & Velocity block (split)  & \texttt{jacobi}          & \texttt{preonly} \\
        $\matAlg{S}$         & Schur complement        & \texttt{gamg}                & \texttt{preonly} \\
        \bottomrule
    \end{tabular}
    \caption{Description of the solvers and preconditioners used for the different blocks of the system \eqref{eq:system-lin-alg}. In monospaced font, we indicate the actual PETSc component used for the solver and preconditioner used. }
    \label{tab:pc_summary}
\end{table}

\subsection{\review{Numerical considerations on the wall shear stress computation}}
\label{sec:comp-framework:wss}

We emphasize that it is difficult to measure WSS experimentally, and numerical simulations are a valuable tool for investigating the complex interactions between the different physical phenomena in the eye, and in particular the WSS distribution~\cite{kumar_numerical_2006,yamamoto_effect_2010,qin_aqueous_2021}.

However, from a computational standpoint, accurately determining the WSS necessitates
\begin{inparaenum}[\it (i)]
    \item a sufficiently fine mesh resolution near the walls to capture the sharp velocity gradients present in these regions, and
    \item a consistent discretization strategy to ensure the accuracy of the computed WSS.
\end{inparaenum}
To achieve the former, we apply the mesh discretization strategy ensuring the mesh is adequately refined near the boundaries of $\Omega_\text{AH}$, as depicted in \cref{fig:coupled:wss_mesh_refinement}.
Regarding the latter, the mesh elements at the interfaces between the domain $\Omega_\text{AH}$ and the surrounding tissues exhibit varying normal vectors, the direction of which are not uniformly defined across all elements.
Since the velocity field is approximated using a $\P_2$ finite element space, its gradient is naturally of order 1, but discontinuous.
Accordingly, we employ a $\P_{1,\mathrm{disc}}$ approximation space to compute the WSS, ensuring consistency and accuracy.
In addition, we compute a continuous piece-wise linear approximation ($[L^2]^3$-projection of the discontinuous approximation) of the WSS to avoid numerical oscillations and improve the visualization of the results.

\begin{figure}
    \centering
    \begin{tikzpicture}
        \node[anchor=south west, inner sep=0] (image) at (0,0) {\includegraphics[width=0.5\textwidth]{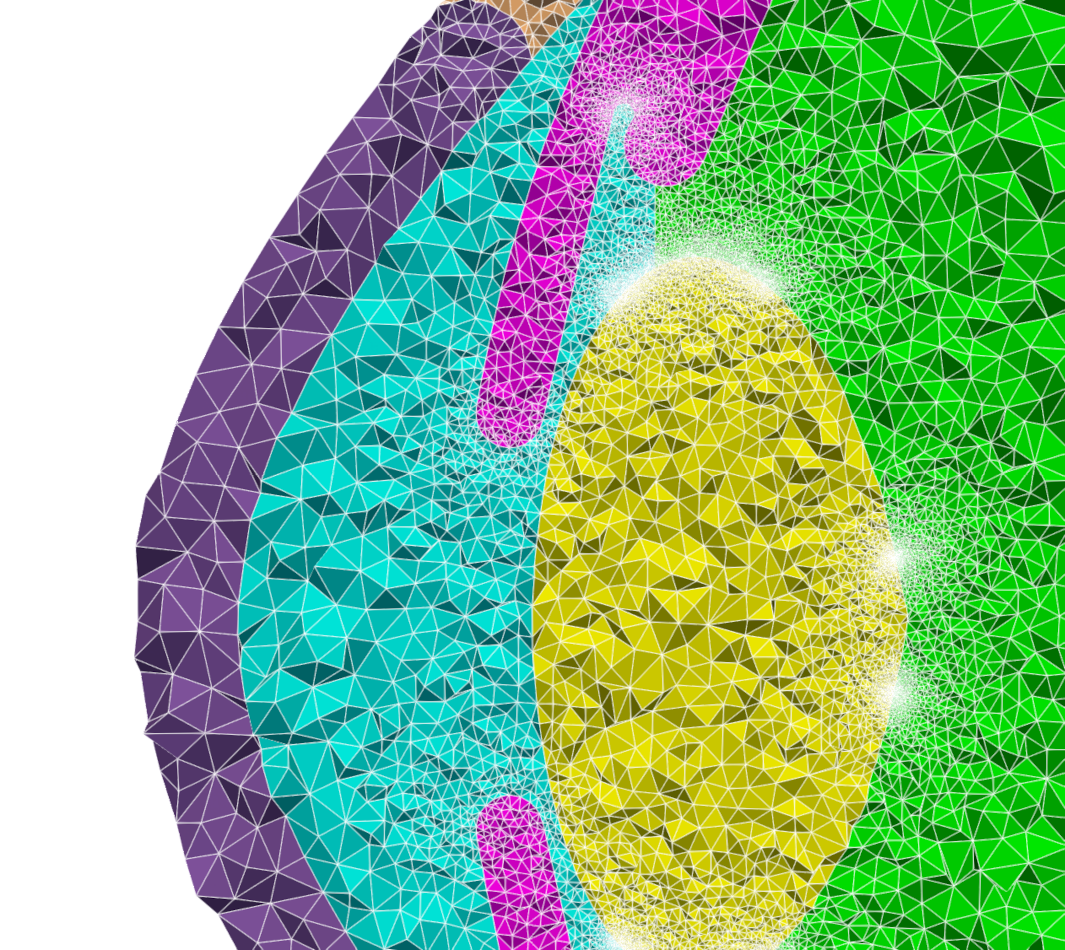}};
        \node[anchor=south west, inner sep=0, draw, color=colorB, line width=2pt] (zoom) at (8,3) {\includegraphics[width=0.2\textwidth]{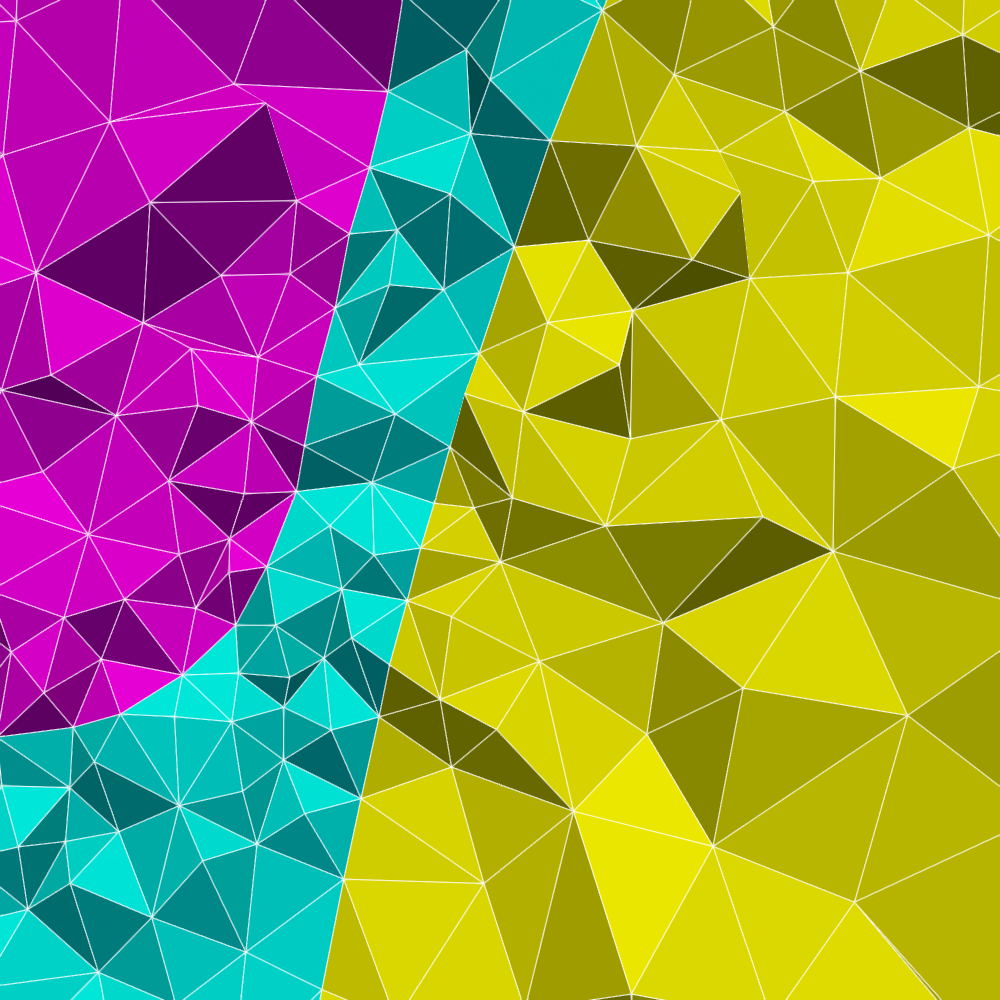}};
        \node[draw=colorB, thick, rectangle, minimum width=0.5cm, minimum height=0.5cm] (rect) at (4.6, 4.3) {};
        \draw[colorB] (rect.north west) -- (zoom.north west);
        \draw[colorB] (rect.south west) -- (zoom.south west);
    \end{tikzpicture}
    \caption{Refinement of the mesh 
    near the wall boundary of $\Omega_\text{AH}$.}
    \label{fig:coupled:wss_mesh_refinement}
\end{figure}

\input{4-numerical-results}
\subsection*{\review{Impact of the posture on the wall shear stress}}
\label{sec:coupled:wss}

We compute the WSS on the surfaces of the anterior chamber to analyze the shear stresses resulting from the AH flow under different postural orientations.
\cref{fig:wss_distribution} displays the WSS magnitude over the corneal endothelium for the standing, prone, and supine positions.
\review{As pointed out in \Cref{sec:comp-framework:wss}, we performed a strategy of discretization to ensure correct refinement near the boundaries of the domain $\Omega_\text{AH}$.
The mesh used in the simulations presented here is the mesh \texttt{Mr5}, described in \cref{tab:coupled:mesh}.
}

\begin{figure}[ht]
    \centering
    \def\subfigwidth{0.25\textwidth}
    \vspace{-0.5cm}
    \newcommand{\figwss}[7][west]{%
\vspace{-0.5cm}
\begin{tikzpicture}
\begin{axis}[
    colorbar,
    colormap={Rainbow Desaturated}{
        rgb255(0cm)=(70,70,219);
        rgb255(0.143cm)=(0,0,91);
        rgb255(0.285cm)=(0,255,255);
        rgb255(0.429cm)=(0,128,0);
        rgb255(0.571cm)=(255,255,0);
        rgb255(0.714cm)=(255,97,0);
        rgb255(0.857cm)=(106,0,0);
        rgb255(1cm)=(223,77,77);
    },
    enlargelimits=false,
    colorbar horizontal,
    point meta min=#3,
    point meta max=#4,
    axis line style = {draw=none},
    tick style = {draw=none},
    xtick = \empty, ytick = \empty,
    colorbar style={
        xlabel = {$\|\vct{\tau}_w\|$ [\si{\Pa}]},
        height=0.05*\pgfkeysvalueof{/pgfplots/parent axis height},
        width=0.9*\pgfkeysvalueof{/pgfplots/parent axis width},
        at={(0.5,-0.02)},
        anchor=center,
        tick label style={font=\footnotesize},
        #6
    },
    width=#5
]
    \addplotgraphicsnatural[xmin=0, xmax=1, ymin=0, ymax=1]{\currfiledir #2}

    \begin{scope}[xshift=0, yshift=90]
        \draw[->] #7 node[midway, anchor=#1] {$\vct{g}$};
    \end{scope}

    \begin{scope}[xshift=95, yshift=90]
        \draw[->, red] (0.1, 0.1) -- (0.09, 0.) node[pos=1, anchor=west] {$x$};
        \draw[->, green!60!black] (0.1, 0.1) -- (0.05, 0.04) node[pos=1, anchor=north east] {$y$};
        \draw[->, blue] (0.1, 0.1) -- (0.02, 0.13) node[pos=1, anchor=east] {$z$};
    \end{scope}

\end{axis}
\end{tikzpicture}
}

\def\subfigwidth{0.39\textwidth}
\subfigure[Standing.\label{fig:eye:heatfluid:wss:standing}]{
    \figwss[north]{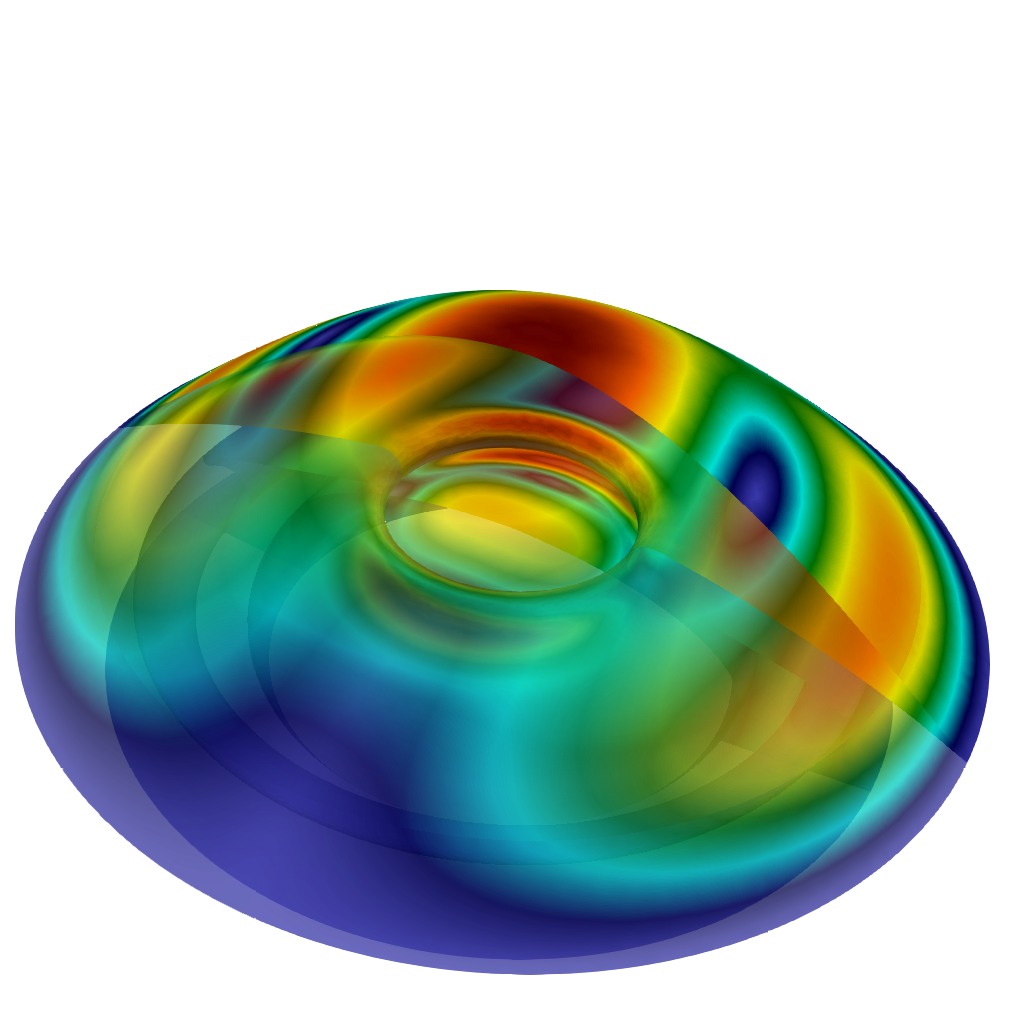}{0}{0.0007718437610900397}{\subfigwidth}{xtick={0,0.0001,0.0002,0.0003,0.0004,0.0005,0.0006,0.0007}}{(0.05, 0.04) -- (0.1, 0.1)}
}
\subfigure[Prone.\label{fig:eye:heatfluid:wss:prone}]{
    \figwss{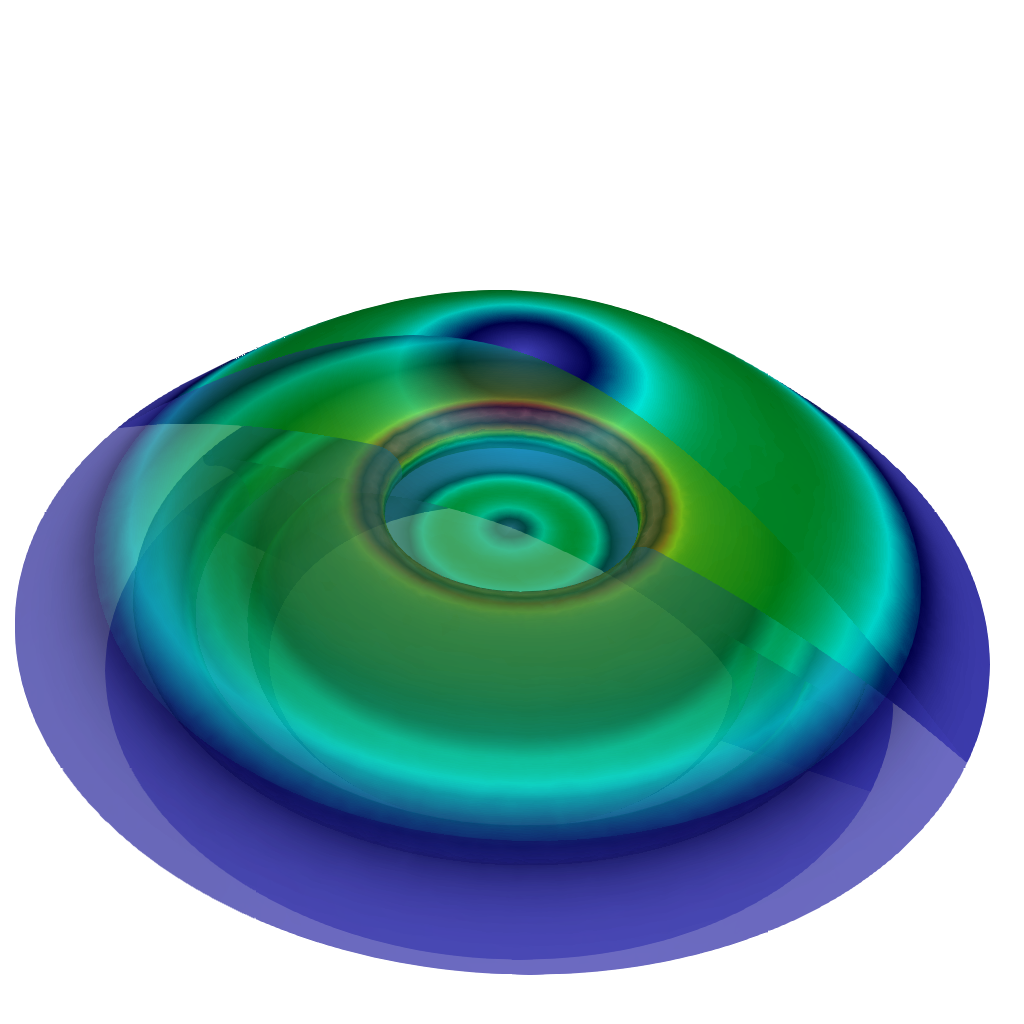}{0}{0.00007180730182807427}{\subfigwidth}{xtick={0,0.00002,0.00004,0.00006,0.00007}}{(0.09, 0.) -- (0.1, 0.1)}
}
\subfigure[Supine.\label{fig:eye:heatfluid:wss:supine}]{
    \figwss{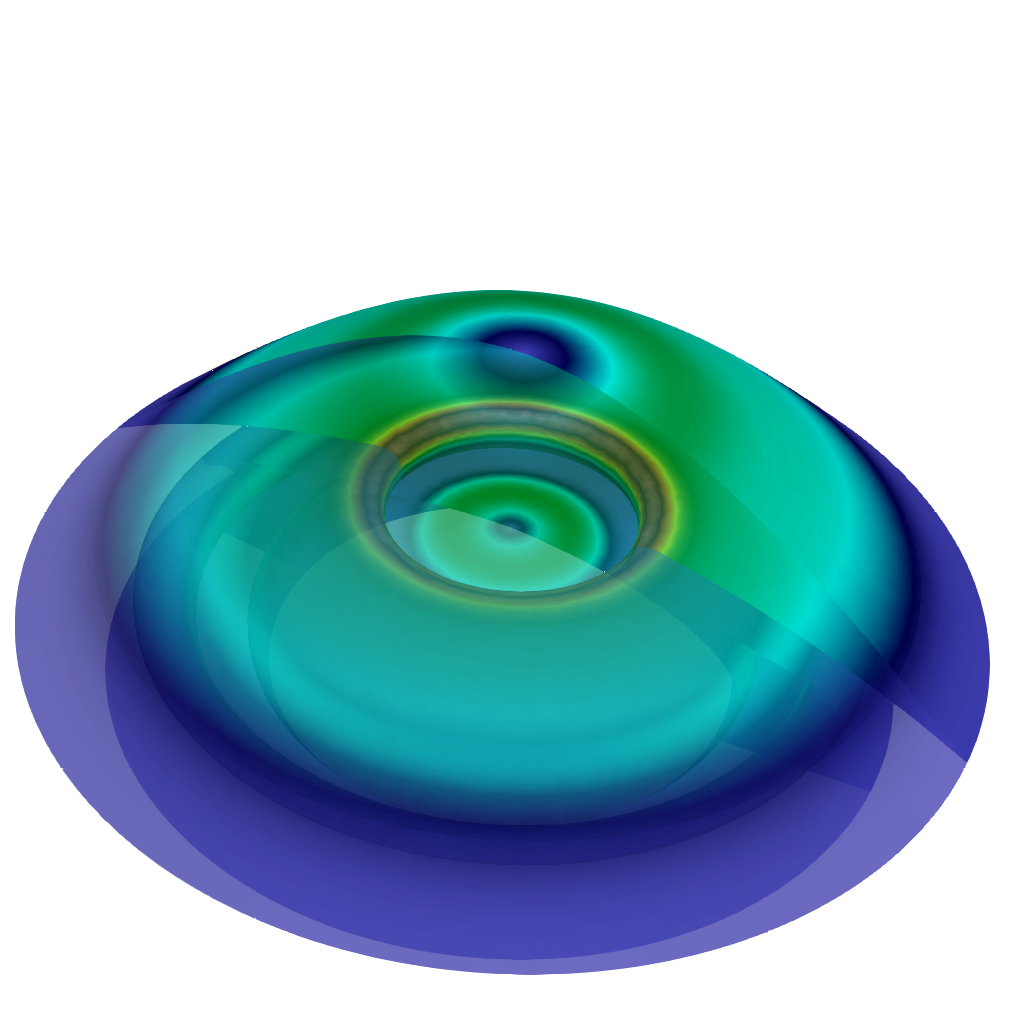}{0}{0.00009524954933823035}{\subfigwidth}{xtick={0,0.00002,0.00004,0.00006,0.00008,0.00009}}{(0.1, 0.1) -- (0.09, 0.)}
}
    \caption{Wall shear stress distribution on the corneal endothelium for the three postural orientations.}
    \label{fig:wss_distribution}
\end{figure}

Comparing the distribution of the WSS, shown in \cref{fig:eye:heatfluid:wss:standing,fig:eye:heatfluid:wss:supine}, with results presented in \cite[Fig.~5]{qin_aqueous_2021}, we observe a similar pattern, with higher values near the corneal endothelium and lower values near the iris and lens surfaces.
\review{
We observe that the WSS distribution varies with different postural orientations. In the standing position (\cref{fig:eye:heatfluid:wss:standing}), the stress is more concentrated towards the lower part of the corneal endothelium, likely due to the effect of gravity.
When in the prone position (\cref{fig:eye:heatfluid:wss:prone}), the stress distribution shifts, showing higher values in the anterior region of the corneal endothelium.
In the supine position (\cref{fig:eye:heatfluid:wss:supine}), the WSS appears more evenly distributed across the corneal endothelium, with a slight concentration towards the posterior region.
These variations highlight the influence of body orientation on the WSS distribution within the eye.
}

\review{As a further refinement and building} on the sensitivity analysis (SA) presented in our previous work \cite{saigre_model_2024}, we next examine the effect of ambient temperature on the WSS magnitude.
Specifically, we calculate the average WSS magnitude across three surfaces within the anterior chamber as a function of ambient temperature for each of the three postural orientations:
\begin{inparaenum}[\it (i)]
    \item on the corneal endothelium, denoted by $\Gamma_\text{cornea}$,
    \item on the iris surface, denoted by $\Gamma_\text{iris}$, and
    \item on the entire boundary of the anterior chamber, denoted by $\partial\Omega_\text{AH}$.
\end{inparaenum}
On the basis of \cite{saigre_model_2024}, we conduct our analysis for a range of $T_\text{amb}$ of $[283, 323]$ \si{\kelvin}, with a baseline value of $T_\text{amb} = \SI{294}{\kelvin}$.

The results are presented in \cref{fig:wss_temperature}.
The first striking result is that the WSS magnitude is significantly influenced by the postural orientation or the subject:
in horizontal positions (prone and supine), the WSS magnitude is ten times lower than in the standing position.
This observation is coherent with the fact the for corneal \review{therapy}, after the injection of endothelial cells inside the aqueous humor or the patient, the patient is placed in prone position for three hours to enhance the adhesion of the cells to the cornea~\cite{kinoshita_injection_2018}. In addition, the results indicate a complex dependency of the WSS magnitude on ambient temperature.
In the three positions, the WSS magnitude reaches a minimum around $T_{\text{amb}} = \qty{310}{\kelvin}$, corresponding to the body temperature. Such a limitation is further discussed in the next section.

In conclusion, our findings show that the WSS magnitude is significantly influenced by both postural orientation and ambient temperature.
These factors should be thus considered in the design of ocular devices and drug delivery systems to optimize therapeutic outcomes.

\begin{figure}[ht]
    \centering
    \newcommand{\figwsstamb}[1]{

\pgfplotstableread[col sep=comma]{\currfiledir #1}\wssdata
\begin{tikzpicture}[scale=0.71]
    \begin{axis}[
    xlabel={$T_\text{amb}$ [\si{\kelvin}]},
    ylabel={$\|\vct{\tau}_w\|$ [\si{\Pa}]},
    xmajorgrids=true,
    ymajorgrids=true,
    legend cell align={left},
    width=0.4\textwidth,
    height=0.55\textwidth,
    ]
    \addplot+ [semithick, colorA, mark=*, mark size=2, mark options={solid}] table[x=TAMB,y=WSS_total] {\wssdata};
    \addlegendentry{$\partial\Omega_\text{AH}$}
    \addplot+ [semithick, colorB, mark=square*, mark size=2, mark options={solid}] table[x=TAMB,y=WSS_cornea] {\wssdata};
    \addlegendentry{$\Gamma_\text{cornea}$}
    \addplot+ [semithick, colorD, mark=x, mark size=2, mark options={solid, line width=1pt}] table[x=TAMB,y=WSS_iris] {\wssdata};
    \addlegendentry{$\Gamma_\text{iris}$}
    \draw[thick, dashed, color=black!60, line width=2pt] (294, \pgfkeysvalueof{/pgfplots/ymin}) -- (294, \pgfkeysvalueof{/pgfplots/ymax});
    \end{axis}
\end{tikzpicture}
}

\subfigure[Prone position.]{
    \figwsstamb{img-wss-data_prone.csv}
}
\subfigure[Supine position.]{
    \figwsstamb{img-wss-data_supine.csv}
}
\subfigure[Standing position.]{
    \figwsstamb{img-wss-data_standing.csv}
}
    \pgfplotstableread[col sep=comma]{\currfiledir img-wss-data_prone.csv}\pronedata
\pgfplotstableread[col sep=comma]{\currfiledir img-wss-data_supine.csv}\supinedata
\pgfplotstableread[col sep=comma]{\currfiledir img-wss-data_standing.csv}\standingdata

\newcommand{\figwsstambbis}[1]{

\begin{tikzpicture}[scale=0.71]
    \begin{axis}[
    xlabel={$T_\text{amb}$ [\si{\kelvin}]},
    ylabel={$\|\vct{\tau}_w\|$ [\si{\Pa}]},
    xmajorgrids=true,
    ymajorgrids=true,
    legend cell align={left},
    width=0.4\textwidth,
    height=0.55\textwidth,
    ]
    \addplot+ [semithick, colorA, mark=*, mark size=2, mark options={solid}] table[x=TAMB,y=#1] {\pronedata};
    \addlegendentry{Prone}
    \addplot+ [semithick, colorB, mark=square*, mark size=2, mark options={solid}] table[x=TAMB,y=#1] {\supinedata};
    \addlegendentry{Supine}
    \addplot+ [semithick, colorD, mark=x, mark size=2, mark options={solid, line width=1pt}] table[x=TAMB,y=#1] {\standingdata};
    \addlegendentry{Standing}
    \draw[thick, dashed, color=black!60, line width=2pt] (294, \pgfkeysvalueof{/pgfplots/ymin}) -- (294, \pgfkeysvalueof{/pgfplots/ymax});
    \end{axis}
\end{tikzpicture}
}

\subfigure[On $\partial\Omega_\text{AH}$.]{
    \figwsstambbis{WSS_total}
}
\subfigure[On $\Gamma_\text{cornea}$.]{
    \figwsstambbis{WSS_cornea}
}
\subfigure[On $\Gamma_\text{iris}$.]{
    \figwsstambbis{WSS_iris}
}
    \caption{Average of the wall shear stress magnitude as a function of the ambient temperature for the three postural orientations, at different location (top row), and impact of the posture on each specific boundaries (bottom row).
    The vertical line at $T_\text{amb} = \SI{294}{\kelvin}$ represents the baseline value as per \cref{tab:coupled:parameters}.}
    \label{fig:wss_temperature}
\end{figure}
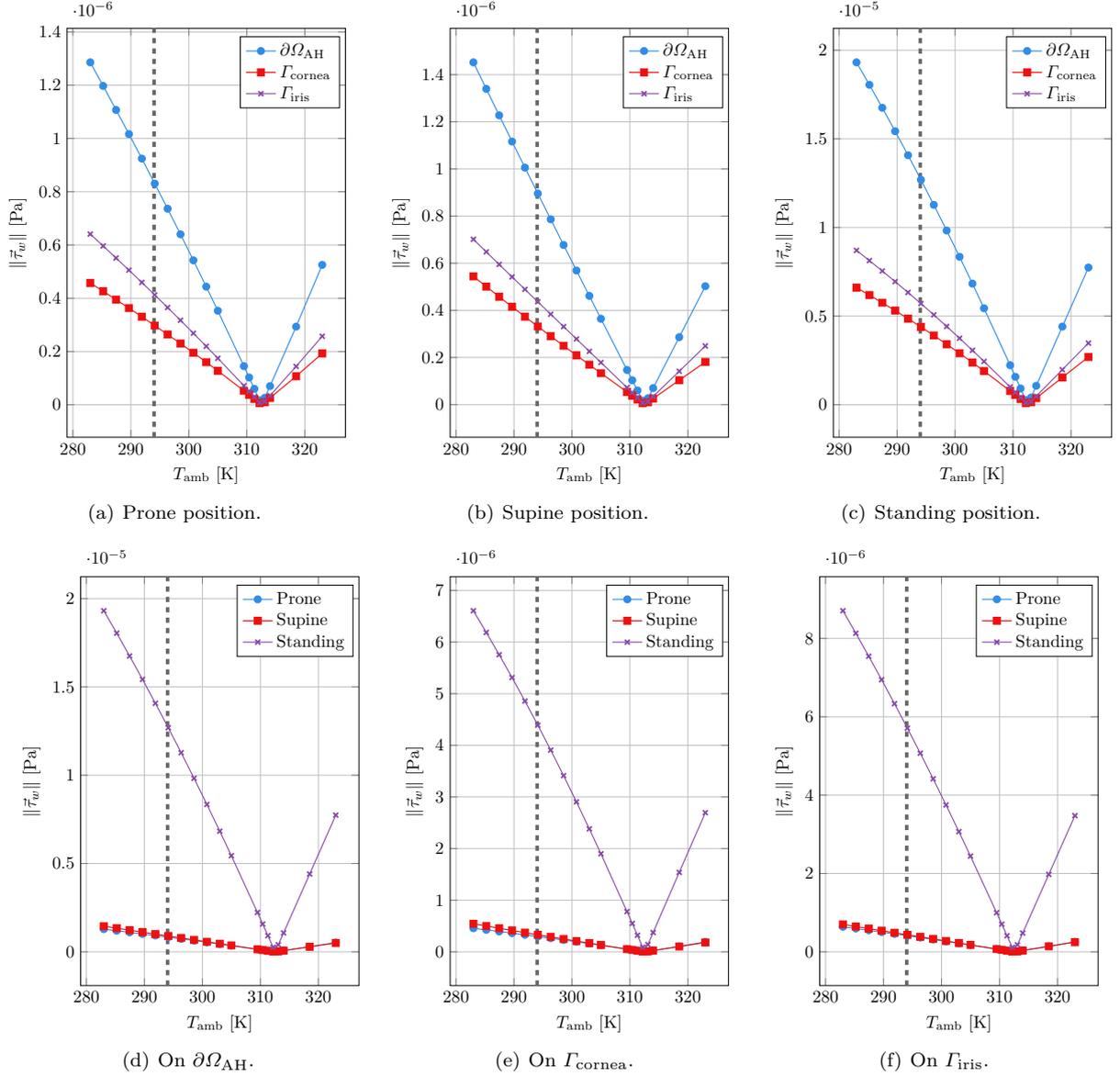

\section{Conclusion and perspectives}
\label{sec:conclusion}

\review{In this work, we have} presented a comprehensive modeling and computational framework for simulating heat transfer within the human eyeball, coupled with the flow of AH in both the anterior and posterior chambers of a healthy eye.
Our complex model has undergone rigorous verification and validation against results from existing literature, demonstrating its accuracy and reliability.
Notably, the model accurately captured the impact of postural orientation on flow recirculations within the eye~\cite{wang_fluid_2016,kilgour_operator_2021,murgoitio-esandi_mechanistic_2023}, providing valuable insights into ocular physiology and the effects of gravity on intraocular fluid dynamics.

Additionally, we have computed the wall shear stress (WSS) distributions within the eye, providing a foundational layer for future applications in \review{therapeutic planning, such as drug delivery or surgery optimization}.
Our findings showed good agreement with previously reported results \cite{fernandez-vigo_computational_2018,kudsieh_numerical_2020,repetto_phakic_2015,qin_aqueous_2021,canning_fluid_2002}.
Moreover, we thoroughly assessed the impact of the postural orientation and of the ambient temperature on the WSS.
This analysis is crucial for understanding the mechanical forces acting on ocular tissues, which can influence drug absorption rates, endothelial cell health, and surgical outcomes~\cite{yang_unraveling_2022,fernandez-vigo_computational_2018}.

An important novelty of our work lies in the integration of high-performance computing (HPC) techniques to solve the coupled heat transfer on the entire eye geometry, with fluid flow equations in both the anterior and posterior chambers.
This holistic approach allows for a more accurate representation of intraocular phenomena compared to models that focus solely on the anterior chamber or neglect the posterior chamber.
By leveraging HPC resources efficiently, we can handle the computational demands of such detailed simulations, enabling high-resolution analyses that were previously impractical.

Despite these successes, a primary drawback of the present model is its computational cost, which remains relatively high—requiring several minutes on the specified hardware to perform a single simulation.
This computational intensity limits the feasibility of real-time simulations, which are desirable for clinical applications and interactive studies.

To address the computational challenges, we are working on improving the preconditioner for the conjugate heat transfer problem, in particular the Schur complement preconditioner for the fluid block.
The challenge is to significantly reduce the computational cost while enable many parameter or real-time evaluations.
We are currently developing model order reduction techniques tailored to our problem, extending our previous work~\cite{saigre_model_2024}.
These methods, such as the (certified) reduced basis method~\cite{prudhomme_reliable_2002}, aim to enable real-time simulations of coupled flow and heat transfer inside the human eyeball by reducing the computational complexity while preserving essential dynamics.
Implementing such techniques will significantly enhance the model's applicability in clinical settings, allowing for rapid simulations that can assist in diagnosis and treatment planning.

Another extension of this work involves incorporating AH inflow and outflow mechanisms, which were neglected in the current model under the assumption of minimal influence on overall flow dynamics, as in \cite{ng_comparative_2008, kumar_numerical_2006}.
Including AH production and drainage would provide a better understanding of intraocular fluid dynamics, especially under pathological conditions such as glaucoma, where these processes are disrupted.
This extension would require modeling AH production in the ciliary body and the trabecular meshwork's drainage function~\cite{guidoboni_mathematical_2019} by means of appropriate boundary conditions \review{complemented with reliable experimental data, in view of calibration and validation. In this direction, boundary conditions involving  pressure for the fluid flow, as employed in \cite{lamminsalo_extended_2018} could be adopted and incorporated in the mathematical and computational framework proposed in the present work, by means of a strategy similar to the one we developed in \cite{bertoluzza_boundary_2017}.}

\review{For the purposes of this contribution, we focused on how temperature gradient enables proper mixing of the fluid in the anterior chamber and on how changes in posture impact the overall dynamics. Whereas the presented model might provide valuable insights on drug delivery in the eye via intracameral injections~\cite{lamminsalo_extended_2018}, a more complex multi-physics description would be required for other delivery routes, such as intravitreal injections~\cite{sadeghi_mathematical_2024}. In particular, the vitreous gel filling the vitreous chamber could be modeled as a viscoelastic fluid, or a more complex coupling with porous media (for the vitreous, neural retina, RPE-choroid, sclera, trabecular meshwork, cornea) described by the Brinkman equations could be employed~\cite{lamminsalo_extended_2020}. In this manner, and provided that experimental data allow parameter calibration, model predictions in the posterior segment of the eye could be enriched to assess whether and to what extent these segments influence the overall dynamics.}

From a clinical perspective, our framework also holds significant potential for assessing the effects of topical administration of ophthalmic drugs, such as eye drops~\cite{bhandari_ocular_2021} or localized hyperthermia treatments~\cite{gongal_thermal_2023}., Additionally, it could be instrumental in evaluating cell injection treatments for internal pathologies, such as bullous keratopathy~\cite{kinoshita_injection_2018}. Future research~\cite{saigre_effect_2025} will focus on integrating drug transport models and cell injection into our simulations, enabling the study of diffusion, absorption, and interaction with ocular tissues.
This integration will facilitate the development of personalized medicine approaches and improve therapeutic strategies by predicting drug efficacy and optimizing dosing regimens.

In summary, this work lays the foundation for advanced computational modeling of ocular fluid dynamics and heat transfer, with promising applications in both research and clinical practice.
The ongoing developments aim to enhance the model's capabilities and usability, bringing us closer to real-time, patient-specific simulations that can inform diagnosis, treatment planning, and potentially device design in ophthalmology.

\section*{Acknowledgments}

The authors would like to acknowledge the support of the Cemosis platform at the University of Strasbourg and the French Ministry of Higher Education, Research and Innovation.

Part of this work was also funded by the France 2030 NumPEx Exa-MA (ANR-22-EXNU-0002) project managed by the French National Research Agency (ANR).


\section*{Material and methods: Conflict of interest statement}

The authors declare no potential conflict of interest.

\section*{Data availability statement}
\label{data}

The geometrical data that support the findings of this study are available in \cite{chabannes_3d_2024}.
All the codes used to run the simulations, as well as the meshed employed in this work, are available in \cite{saigre_mesh_2024}.

\printbibliography{}
\addcontentsline{toc}{section}{References}

\end{document}